\documentclass[12pt,a4paper]{amsart}
\usepackage{amsfonts}
\usepackage{amsthm}
\usepackage{amsmath}
\usepackage{amscd}
\usepackage[english]{babel}
\usepackage[latin2]{inputenc}
\usepackage{graphicx}
\usepackage[square,numbers]{natbib}
\usepackage{tikz}
\numberwithin{equation}{section}
\usepackage[margin=2.9cm]{geometry}
\usepackage{dsfont}
\usepackage{enumitem}
\usepackage{color}
\usepackage{lineno}
\usepackage{hyperref}

\allowdisplaybreaks

\makeatletter
\newcommand{\xRightarrow}[2][]{\ext@arrow 0359\Rightarrowfill@{#1}{#2}}
\makeatother

\makeatletter
\newcommand*{\rom}[1]{\expandafter\@slowromancap\romannumeral #1@}
\makeatother

\allowdisplaybreaks

\theoremstyle{plain}
\newtheorem{Th}{Theorem}[section]
\newtheorem{Lemma}[Th]{Lemma}
\newtheorem{Cor}[Th]{Corollary}

\theoremstyle{definition}
\newtheorem{Def}[Th]{Definition}
\newtheorem{Rem}[Th]{Remark}
\newtheorem{Ex}[Th]{Example}

\begin{document}

	\title[]{On weak conditional convergence of bivariate Archimedean and Extreme Value copulas, and consequences to nonparametric estimation}
	\author[]{Thimo M. Kasper, Sebastian Fuchs, Wolfgang Trutschnig}
	\address{Department for Mathematics, University of Salzburg \\
Hellbrunnerstrasse 34, A-5020 Salzburg, Austria}
	\email{thimo.kasper@sbg.ac.at \\ sebastian.fuchs@sbg.ac.at \\ wolfgang.trutschnig@sbg.ac.at}


\maketitle
\vspace*{-0.6cm}
\section*{Keywords}
Archimedean copula,
Extreme Value copula,
Checkerboard copula,
weak convergence,
estimation,
dependence measure
		
		\begin{abstract}
			Looking at bivariate copulas from the perspective of conditional distributions and considering weak convergence of almost 
			all conditional distributions yields the notion of weak conditional convergence.  
			At first glance, this notion of convergence for copulas might seem far too restrictive to be of any practical 
			importance - in fact, given samples of a copula $C$ the corresponding empirical co\-pulas do not converge 
			weakly conditional to $C$ with probability one in general.  
			Within the class of Archimedean copulas and the class of Extreme Value copulas, however, standard pointwise convergence and weak conditional convergence can even be proved to be equivalent. Moreover, it can be shown that every copula $C$ is the weak
			conditional limit of a sequence of checkerboard copulas. After proving these three main results 
			and pointing out some consequences we sketch some implications for  
			two recently introduced dependence measures and for the nonparametric estimation of Archimedean and Extreme Value copulas.    
		\end{abstract}


	\section{Introduction}
	Suppose that $\{C_\theta: \theta \in \Theta\}$ is a parametric class of bivariate copulas with 
	$\Theta \subseteq \mathbb{R}^d$ for some $d \in \mathbb{N}$ 
	and let $\{K_\theta: \theta \in \Theta\}$ denote the corresponding conditional distributions (Markov kernels), i.e., 
	if $X,Y$ are uniformly distributed on $[0,1]$ and $(X,Y)$ has distribution function $C_\theta$ 
	then $K_\theta(x,E)=\mathbb{P}(Y \in E \vert X=x)$.  
	Many standard classes of copulas are not only continuous in the parameter with respect to pointwise/uniform convergence 
	(see \cite{Principles, Nelsen}) but exhibit the even stronger property that if $(\theta_n)_{n \in \mathbb{N}}$ 
	converges to 
	$\theta$ then almost all conditional distributions $(K_{\theta_n}(x,\cdot))_{n \in \mathbb{N}}$ converge weakly to $K_\theta(x,\cdot)$. In the sequel we will refer to weak convergence of almost all conditional distributions as weak conditional convergence.
	It is straightforward to verify that (among many others) the family of Gaussian co\-pulas and the family of $t$-copulas 
	exhibit the just mentioned continuity with respect to the parameter. Moreover, leaving the absolutely continuous setting, the same is true, e.g., for the Marshall-Olkin family.   
	
	Despite the afore-mentioned examples, at first glance, weak conditional convergence might seem as a concept far too restrictive to be of any practical importance outside the purely parametric setting. 
	This impression is reinforced by the fact that given samples $(X_1,Y_1),(X_2,Y_2),\ldots$ from a copula $C$ and letting $\hat{E}_n$ denote the corresponding empirical copula (bivariate interpolation of the induced subcopula, see \cite{Nelsen}) we do not have weak conditional convergence of $(\hat{E}_n)_{n \in \mathbb{N}}$ to $C$ unless $C$ is completely dependent in the sense that random variable $Y$ is a measurable function of random variable $X$ (see \cite{lanc}).  
	
	As we will demonstrate in this contribution, however, within the class of Archimedean copulas and the class of Extreme Value copulas (neither of them being a parametric class of the afore-mentioned type) standard pointwise/uni\-form convergence and weak conditional convergence are even equivalent, a result having direct implications for the dependence measures $\zeta_1$ and $r$ introduced in \cite{p06} and \cite{Dette}, respectively, as well as for the nonparametric estimation of Archimedean and Extreme Value copulas 
	(see \cite{barbe,GNZ,GenestInference} and \cite{InferenceEVs, GS}). 
	We will show that convexity of the univariate \textquoteleft generating' functions (the normalized generator in the Archimedean and the Pickands dependence function in the Extreme Value case) is the key property entailing weak conditional convergence. Additionally, building upon the theorems in \cite{approximation} we will derive a universal approximation result with respect to weak conditional convergence and show that for every bivariate copula $C$ we can find a sequence 
	$(C_n)_{n \in \mathbb{N}}$ of checkerboard copulas that converges weakly conditional to the copula $C$.
	
	The authors' interest in studying convergence of Archimedean copulas was triggered by \cite{convArch} where the authors among other things showed that pointwise/uniform convergence of a sequence of Archimedean copulas to an Archimedean copula is equivalent to convergence of the corresponding sequence of Kendall distribution functions. In our contribution we first
	derive a slightly modified version of this result (including the fact that we can have convergence of the copulas without having convergence of the corresponding generators in $0$, see Theorem \ref{T1}) 
	and then go one step further (see Theorem \ref{result}) and prove the equivalence of six different notions of convergence (some involving the copulas, some the generators), weak conditional convergence being one of them.
	
	The rest of this contribution is organized as follows: Section 2 gathers preliminaries and notations that will be used throughout 
	the paper. In Section 3 we formally define weak conditional convergence, prove that checkerboard copulas are dense with respect to weak conditional convergence, and show that weak conditional convergence implies convergence with 
	respect to the metric $D_1$ introduced in \cite{p06} but 
	in general not vice versa. Section 4 derives 
	the afore-mentioned equivalence of pointwise/uniform and weak conditional convergence within the family of Archimedean copulas
	in several steps. In Section 5 we prove an analogous characterization of convergence within the class of Extreme Value copulas. 
	Direct consequences of these two main results to the estimation of Archimedean and Extreme Value copulas are sketched and illustrated via simulations in Section 6. 
	Finally, we use the obtained results for estimating the recently introduced coefficient of correlation (see \cite{Chatterjee}) and compare the performance of the estimators incorporating or ignoring the Extreme Value/Archimedean information.
		
	\section{Notation and preliminaries}
	In the sequel we will let $\mathcal{C}$ denote the family of all bivariate copulas. For each copula $C$ the corresponding doubly stochastic measure will be denoted by $\mu_C$, i.e. $\mu_C([0,x]\times[0,y]) = C(x,y)$ for all $x,y \in [0,1]$, $\mathcal{P}_\mathcal{C}$ will denote the family of all doubly stochastic measures. For more background on copulas and doubly stochastic measures we refer to \cite{Principles,Nelsen}. 
	For every metric space $(S,d)$ the Borel $\sigma$-field on $S$ will be denoted by $\mathcal{B}(S)$.
	
	In what follows Markov kernels will play a prominent role. A \emph{Markov kernel} from $\mathbb{R}$ to $\mathbb{R}$ is a mapping $K: \mathbb{R}\times\mathcal{B}(\mathbb{R}) \rightarrow [0,1]$ such that for every fixed $E\in\mathcal{B}(\mathbb{R})$ the mapping $x\mapsto K(x,E)$ is (Borel-)measurable and for every fixed $x\in\mathbb{R}$ the mapping $E\mapsto K(x,E)$ is a probability measure. Given two real-valued random variables $X,Y$ on a probability space $(\Omega, \mathcal{A}, \mathbb{P})$ we say that a Markov kernel $K$ is a \emph{regular conditional distribution} of $Y$ given $X$ if 
	$K(X(\omega), E) = \mathbb{E}(\mathds{1}_E \circ Y | X)(\omega) $
	holds $\mathbb{P}$-almost surely for every $E\in \mathcal{B}(\mathbb{R})$. It is well-known (see, e.g., \cite{Kallenberg,Klenke}) that for $X,Y$ as above, a regular conditional distribution of $Y$ given $X$ always exists and is unique for $\mathbb{P}^X$-a.e. $x\in\mathbb{R}$. If $(X,Y)$ has distribution function $H$ (in which case we will also write $(X,Y) \sim H$ and let $\mu_H$ denote the corresponding probability measure on $\mathcal{B}(\mathbb{R}^2)$) we will let $K_H$ denote (a version of) the regular conditional distribution of $Y$ given $X$ and simply refer to it as \emph{Markov kernel of $H$}. If $C$ is a copula then we will consider the Markov kernel of $C$ automatically as mapping $K_C:[0,1] \times\mathcal{B}([0,1])\rightarrow [0,1]$.
	Defining the $x$-section of a set $G\in\mathcal{B}(\mathbb{R}^2)$ as $G_x := \lbrace y\in \mathbb{R}: (x,y)\in G\rbrace$ the so-called disintegration theorem (see \cite{Kallenberg,Klenke}) yields 
	\begin{align}\label{eq:di}
	\int\limits_{\mathbb{R}} K_H(x,G_x) \ \mathrm{d}\mathbb{P}^X(x) = \mu_H(G).
	\end{align}
	As a direct consequence, for every $C \in \mathcal{C}$ we get
	\begin{align*}
	\int\limits_{[0,1]} K_C(x,E) \ \mathrm{d}\mathbb{P}^X(x) = \int\limits_{[0,1]} K_C(x,E) \ \mathrm{d}\lambda(x) = \lambda(E)
	\end{align*}
	for every $E\in \mathcal{B}([0,1])$, whereby $\lambda$ denotes the Lebesgue measure on $\mathbb{R}$. For more background 
	on conditional expectation and general disintegration we refer to \cite{Kallenberg, Klenke}. 

	We call a copula $C$ \emph{completely dependent} if there exists a $\lambda$-preserving transformation $h:[0,1] \to [0,1]$ (i.e., a transformation fulfilling $\lambda(h^{-1}(E))=\lambda(E)$ for every $E \in \mathcal{B}([0,1])$) such that
	$K(x,E):=\mathbf{1}_E(h(x))$ is a Markov kernel of $C$. For more properties of complete dependence we refer to \cite{lanc} as well as to \cite{p06} and the references therein.
	
	Markov kernels can be used to define metrics stronger than the standard \emph{uniform metric} $d_\infty$, defined by
	\begin{align}
	d_\infty(C_1,C_2) := \max\limits_{(x,y)\in[0,1]^2} |C_1(x,y) - C_2(x,y)|,
	\end{align}
	on $\mathcal{C}$. 
	It is well known that the metric space $(\mathcal{C}, d_\infty)$ is compact and that pointwise and uniform convergence 
	of a sequence of copulas $(C_n)_{n\in \mathbb{N}}$ are equivalent (see \cite{Principles, p07}). 
	Following \cite{p06} and defining
	\begin{align}
	D_1(C_1,C_2) &:= \int\limits_{[0,1]}\int\limits_{[0,1]} |K_{C_1}(x,[0,y]) - K_{C_2}(x,[0,y])| \ \mathrm{d}\lambda(x) \mathrm{d}\lambda(y), \\
	D_2^2(C_1,C_2) &:= \int\limits_{[0,1]}\int\limits_{[0,1]} (K_{C_1}(x,[0,y]) - K_{C_2}(x,[0,y]))^2 \ \mathrm{d}\lambda(x) \mathrm{d}\lambda(y), \nonumber\\
	D_\infty(C_1,C_2) &:= \sup_{y\in[0,1]}\int\limits_{[0,1]} |K_{C_1}(x,[0,y]) - K_{C_2}(x,[0,y])| \ \mathrm{d}\lambda(x)\nonumber
	\end{align}
	it can be shown that $D_1,D_2,D_\infty$ are metrics generating the same topology on $\mathcal{C}$. 
	In the sequel we will mainly work with $D_1$ and refer to \cite{p15, p06} for more information on $D_2$ and $D_\infty$. The metric space $(\mathcal{C}, D_1)$ is complete and separable but not compact. Moreover, we have 
	$D_1(C,\Pi)\in [0,\frac{1}{3}]$ for every $C \in \mathcal{C}$ and $D_1(C,\Pi)$ is maximal if and only if $C$ is completely dependent.
	The metric $D_1$ was originally introduced in order to construct a dependence measure which, contrary to $d_\infty$, is capable of separating independence and complete dependence. The resulting $D_1$-based dependence measure $\zeta_1$ introduced in \cite{p06} is defined as  
	\begin{align}
	\zeta_1(C) := 3\cdot D_1(C, \Pi)
	\end{align}
	for every $C\in\mathcal{C}$. In the sequel we will also consider the dependence measure $r(X,Y)=r(C)$ 
	introduced in \cite{Dette} as   
	\begin{align}
	r(C) := 6\cdot \int\limits_{[0,1]}\int\limits_{[0,1]} K_C(x,[0,y])^2 \mathrm{d}\lambda(x) \mathrm{d}\lambda(y) - 2.
	\end{align}
	It is straightforward to verify that $r(C)$ can be expressed in terms of $D_2$ and that 
	$r(C)= 6\cdot D_2^2(C,\Pi)$ holds. Both $\zeta_1$ and $r$ attain values in $[0,1]$, are $0$ if, and only if $C=\Pi$, and $1$ if, and only if $C$ is completely dependent. 
	
  	A symmetric version of $D_1$-convergence (or, equivalently, $D_2$-converence) was introduced by Mikusi\'{n}ski 
  	and Taylor in \cite{approxNCopulas} under the name $\partial$-convergence:
	A sequence $(C_n)_{n \in \mathbb{N}}$ $\partial$-converges to a copula $C \in \mathcal{C}$
	($C_n \xrightarrow{\partial} C$ for short)
	if and only if 
	\begin{align*}
		\lim\limits_{n\to\infty} D_1(C_n,C) + D_1(C_n^t, C^t) = 0.
	\end{align*}
	We will see in the next section that in general weak conditional convergence does not imply $\partial$-converges 
	(Example \ref{wcc.symm.ex}) nor vice versa. For a thorough survey of different notions on convergence of copulas we refer to \cite{Sempi:Convergence}. 
		
	\section{Weak conditional convergence and checkerboards}
	Sticking to the idea of viewing bivariate copulas in terms of their conditional distributions and considering weak convergence gives rise to what we refer to as weak conditional convergence in the sequel: 
	\begin{Def}\label{def:wcc}
		Suppose that $C,C_1,C_2,\ldots$ are copulas and let $K_C,K_{C_1}, K_{C_2},\ldots$ be (versions of) the corresponding Markov kernels. We will say that $(C_n)_{n \in \mathbb{N}}$ converges \emph{weakly conditional} 
		to $C$ if and only if for $\lambda$-almost every $x \in [0,1]$ we have that the sequence $(K_{C_n}(x,\cdot))_{n \in \mathbb{N}}$ of probability measures on $\mathcal{B}([0,1])$ converges weakly to the probability measure $K_{C}(x,\cdot)$.
		In the latter case we will write $C_n \xrightarrow{\text{wcc}} C$ (where \textquoteleft wcc' stands for 
		\textquoteleft weak conditional convergence'). 
	\end{Def}
	As already mentioned in the Introduction, many standard parametric classes  $\{C_\theta: \theta \in \Theta\}$  of copulas 
	depend on the parameter $\theta$ weakly conditional in the sense that if 
	$(\theta_n)_{n \in \mathbb{N}}$ converges to $\theta$ then the corresponding sequence $(C_{\theta_n})_{n \in \mathbb{N}}$ converges weakly conditional to $C_\theta$.  
	This is obviously true for parametric classes $\{C_\theta: \theta \in \Theta\}$ of absolutely continuous copulas 
	whose corresponding densities $\{k_\theta: \theta \in \Theta\}$ have the property that if 
	$(\theta_n)_{n \in \mathbb{N}}$ converges to $\theta \in \Theta$ then $(k_{\theta_n})_{n \in \mathbb{N}}$ converges to 
	$k_\theta$ $\lambda_2$-almost everywhere (whereby $\lambda_2$ denotes the two-dimensional Lebesgue measure on $\mathcal{B}(\mathbb{R}^2)$). In fact, in this case the corresponding Markov kernels $K_{C_{\theta_n}}$ are given by
	\begin{align}
	K_{C_{\theta_n}}(x,[0,y]) = \int_{[0,y]} k_{\theta_n}(x,s) \, d\lambda(s)
	\end{align}
	and if we let $\Lambda \in \mathcal{B}([0,1]^2)$ denote the set of all points $(x,y)$ fulfilling 
	$\lim_{n \rightarrow \infty} k_{\theta_n}(x,y)=k_\theta(x,y)$ then disintegration yields $\lambda(\Lambda_x)=1$ for 
	$\lambda$-almost every $x \in [0,1]$, so the property follows immediately. 
	It is straightforward to verify that (among many others) the family of Gaussian copulas and the family of $t$-copulas fulfill this property.
	
	The same is true for other, not necessarily absolutely continuous classes like the Marshall-Olkin family 
	$(M_{\alpha,\beta})_{(\alpha,\beta) \in [0,1]^2}$ (see \cite{MarshallOlkinDistributions,Principles,Nelsen}) given by 
	\begin{align}
	M_{\alpha,\beta} (x,y)  =  \begin{cases}
	x^{1-\alpha} \, y & \text{if } x^\alpha \geq y^\beta\\
	x\,y^{1-\beta} & \text{if } x^\alpha < y^\beta.
	\end{cases}
	\end{align} 
	According to \cite{p06} the corresponding Markov kernel $K_{M_{\alpha,\beta}}$ is given by 
	\begin{align}
	K_{M_{\alpha,\beta}} (x,[0,y])  = \begin{cases}
	(1-\alpha) x^{-\alpha}\,y & \text{if } y^\beta < x^\alpha\\
	y^{1-\beta} & \text{if }y^\beta \geq  x^\alpha
	\end{cases}
	\end{align} 
	and it is straightforward to verify that if the parameter vector $(\alpha_n,\beta_n)_{n \in \mathbb{N}}$ converges to 
	$(\alpha,\beta)$ then we also have $M_{\alpha_n,\beta_n} \xrightarrow{\text{wcc}} M_{\alpha,\beta}$.
	
	Before focusing on the Archimedean and the Extreme Value setting we prove a general approximation result saying that 
	the class of checkerboard copulas is dense in $\mathcal{C}$ w.r.t. 
	weak conditional convergence (see Theorem \ref{thm:checkerboard}).
	Recall that a copula $C$ is called a \emph{checkerboard copula} with resolution $N \in \mathbb{N}$ if and only if $\mu_C$ distributes its mass uniformly on each rectangle $R_{ij}^N=[\frac{i-1}{N},\frac{i}{N}] \times [\frac{j-1}{N},\frac{j}{N}]$ with $i,j \in \{1,\ldots,N\}$. We will refer to $\mathcal{S}_N$ as the family of all checkerboard copulas with resolution $N$, 
	the set $\mathcal{S}=\bigcup_{N=1}^\infty \mathcal{S}_N$ is called the class of all checkerboard copulas (checkerboards for short). 
	For every $N \in \mathbb{N}$ the (unique) checkerboard copula $\mathfrak{CB}_{N}(C) \in \mathcal{S}_N$ fulfilling 
	$$
	\mu_C(R_{ij}^N)=\mu_{\mathfrak{CB}_{N}(C)}(R_{ij}^N)
	$$
	for all $i,j \in \{1,\ldots,N\}$ will be referred to as \emph{$N$-checkerboard approximation of the copula $C$}. 
	
	It is well-known that the class of checkerboard copulas $\mathcal{S}$ is dense in $(\mathcal{C},d_\infty)$ and in $(\mathcal{C},D_1)$, see \cite{Principles, approximation, p06}. The following theorem implies these two interrelations: 
	\begin{Th}\label{thm:checkerboard}
		Given a copula $C$, there is a sequence
	of elements of $\mathcal{S}$ that converges to $C$ with respect to weak conditional convergence. In other words, $\mathcal{S}$ is dense in $\mathcal{C}$ with respect to weak conditional convergence.
	\end{Th} 
	\begin{proof}
		Fix $C \in \mathcal{C}$ and suppose that $K_C$ is a Markov kernel of $C$. We are going to show that $(\mathfrak{CB}_{2^n}(C))_{n \in \mathbb{N}}$ converges weakly conditional to $C$.
		Using Lipschitz continuity of copulas in each coordinate, for every $y$ of the form $y=\frac{j}{2^m}$ with $m \in \mathbb{N}$ and $j \in \{0,\ldots,2^m\}$ there exists a set $\Lambda_y \in \mathcal{B}([0,1])$ with the following three properties: 
		\begin{enumerate}
			\item For every $x \in \Lambda_y$ the function $t \mapsto C(t,y)$ is differentiable at $x$ and fulfills 
			$\frac{\partial C}{\partial x}(x,y)=K_C(x,[0,y])$,
			\item $\lambda(\Lambda_y)=1$,
			\item $\Lambda_y \subseteq \Big(\bigcup_{l=1}^\infty \{0,\frac{1}{2^l},\frac{2}{2^l},\ldots,\frac{2^l-1}{2^l},1\}\Big)^c$.
		\end{enumerate}
		Letting $\mathcal{W}$ denote the set of all points $y$ of the form $y=\frac{j}{2^m}$ with $m \in \mathbb{N}$ and 
		$j \in \{0,\ldots,2^m\}$ it follows that $\mathcal{W}$ is countably infinite, hence setting 
		$\Lambda=\bigcap_{y \in \mathcal{W}} \Lambda_y$ yields $\lambda(\Lambda)=1$. \\
		Consider $y=\frac{j}{2^m} \in \mathcal{W}$ and $x \in \Lambda$. For every $n \geq m $ there exists exactly one index 
		$i_n(x) \in \{0,\ldots,2^n\}$ with  
		$$
		x \in \left(\frac{i_n(x)-1}{2^n}, \frac{i_n(x)}{2^n} \right). 
		$$    
		Considering that $t \mapsto K_{\mathfrak{CB}_{2^n}(C)}(t,\cdot)$ is constant on the interval $\left(\frac{i_n(x)-1}{2^n}, \frac{i_n(x)}{2^n} \right)$ disintegration yields
		\begin{align*}
		C\left(\frac{i_n(x)}{2^n},y \right) - C\left(\frac{i_n(x)-1}{2^n},y \right) &=
		\int\limits_{\left(\frac{i_n(x)-1}{2^n}, \frac{i_n(x)}{2^n} \right]} K_{\mathfrak{CB}_{2^n}(C)}(t,[0,y]) \mathrm{d}\lambda(t) \\
		&= \frac{1}{2^n} \, K_{\mathfrak{CB}_{2^n}(C)}(x,[0,y]), 
		\end{align*}
		from which we directly get
		$$
		K_{\mathfrak{CB}_{2^n}(C)}(x,[0,y]) = \frac{C\left(\frac{i_n(x)}{2^n},y \right) - C\left(\frac{i_n(x)-1}{2^n},y \right)}{\frac{1}{2^n}} \stackrel{n \rightarrow \infty}{\longrightarrow} 
		\frac{\partial C}{\partial x}(x,y)=K_C(x,[0,y]).
		$$
		Since $(x,y) \in \Lambda \times \mathcal{W}$ was arbitrary we have shown that for each $x \in \Lambda$ the conditional distribution functions $y \mapsto K_{\mathfrak{CB}_{2^n}(C)}(x,[0,y])$ converge to $y \mapsto K_C(x,[0,y])$ for every 
		$y \in \mathcal{W}$, i.e., on a dense set. Having this, weak conditional convergence of $(\mathfrak{CB}_{2^n}(C))_{n \in \mathbb{N}}$ to $C$ 
		follows immediately. 
	\end{proof}

	We conclude this section with an example clarifying the interrelation between the afore-mentioned notions of convergence: According to Lemma 7 in \cite{p06} weak conditional convergence of $(C_n)_{n \in \mathbb{N}}$ to $C$ implies convergence w.r.t. $D_1$. Additionally, convergence w.r.t. $D_1$ implies convergence in $d_\infty$ but not vice versa. 
	The following simple example shows that we can have $D_1$-convergence without having weak conditional convergence.

	\begin{Ex}\label{ex:d1.not.wcc}
		Let $N\in\mathbb{N}, i\in\lbrace 1,\ldots,N \rbrace$ and $n = 2^N+i-2$. Define the copula $C_n$ via its Markov kernel by
		\begin{align*}
		K_{C_n}(x,[0,y]) = \begin{cases}
		\mathds{1}_{[0,y]}(2^N\cdot x  + 1 - i) & \text{ if } x\in [\frac{i-1}{2^N},\frac{i}{2^N}] \\
		y & \text{ if } x\in [0,1]\setminus[\frac{i-1}{2^N},\frac{i}{2^N}].
		\end{cases}
		\end{align*}
		Then $C_n$ does not converge weakly conditional to $\Pi$ since for every $x \in [0,1]$ the probability measure
		$K_{C_n}(x,\cdot)$ is, on the one hand, degenerated for infinitely many $n \in \mathbb{N}$ and, on the other hand, 
		coincides with $\lambda$ restricted to $[0,1]$ infinitely many times too. 
		Nevertheless $C_n$ converges to $\Pi$ w.r.t. $D_1$ since 
		\begin{align*}
		D_1(C_n,\Pi) \leq \frac{1}{2^N}
		\end{align*}
		holds and if $n$ goes to infinity then so does $N$. 
	\end{Ex}
	
	We conclude this section with an example showing that, in contrast to convergence w.r.t. to $d_\infty$,
	neither convergence w.r.t. $D_1$ nor weak conditional convergence is a symmetric concept, 
	i.e., considering Markov kernels $K_{C_n^t}(x,\cdot)$ instead of $K_{C_n}(x,\cdot), n\in\mathbb{N}$,  
	may yield a different notion of convergence (as usual $C^t$ denotes the transpose of $C$, i.e. $C^t(x,y) = C(y,x)$).

	\begin{Ex} \label{wcc.symm.ex}
	    For every $n \in \mathbb{N}$ define the $\lambda$-preserving transformation 
		$h_n: [0,1] \to [0,1]$ by 
		$$
		  h_n (x) 
			:= 2^n \, x \, (mod \, 1)
		$$
		and let $C_n$ denote the corresponding completely dependent copula.
		Since $K_{C_n^t} (x,.)$ is the (discrete) uniform distribution on 
		$\{\frac{x}{2^n} + \frac{i}{2^n} \, : \, i \in \{0, \dots, 2^n-1\}\}$, 
		$K_{C_n^t} (x,.)$ converges weakly to $K_{\Pi} (x,.)$ for $\lambda$-almost every $x \in [0,1]$ and we have  
		$$ 
		  C_n^t \xrightarrow{wcc} \Pi,
		$$
		implying that $(C_n^t)_{n \in \mathbb{N}}$ converges to $\Pi$ w.r.t. $D_1$.
		Additionally (see \cite[Theorem 14]{p06}) $D_1 (C_n,\Pi) = 1/3$ holds for every $n \in \mathbb{N}$ from which it follows
		immediately that $(C_n)_{n \in \mathbb{N}}$ does not converge to $\Pi$ neither weakly conditional nor w.r.t. $D_1$.
		It can even be shown that $(C_n)_{n \in \mathbb{N}}$ does not converge in $(\mathcal{C},D_1)$ at all: 
		If $(C_n)_{n \in \mathbb{N}}$ would converge to some copula $C$ w.r.t. $D_1$ then 
		according to \cite[Proposition 15]{p06}) $C$ would be completely dependent, i.e., there would be some 
		$\lambda$-preserving transformation $h$ such that $C=C_h$ holds.  
		Considering that $D_1$-convergence implies $d_\infty$-convergence $C_h^t=\Pi$ would follow, a contradiction since 
		$C_h^t$ is singular whereas $\Pi$ is absolutely continuous. 
	\end{Ex}
	
	\begin{figure}[htp!]
			\centering
			\includegraphics[width=0.8\textwidth]{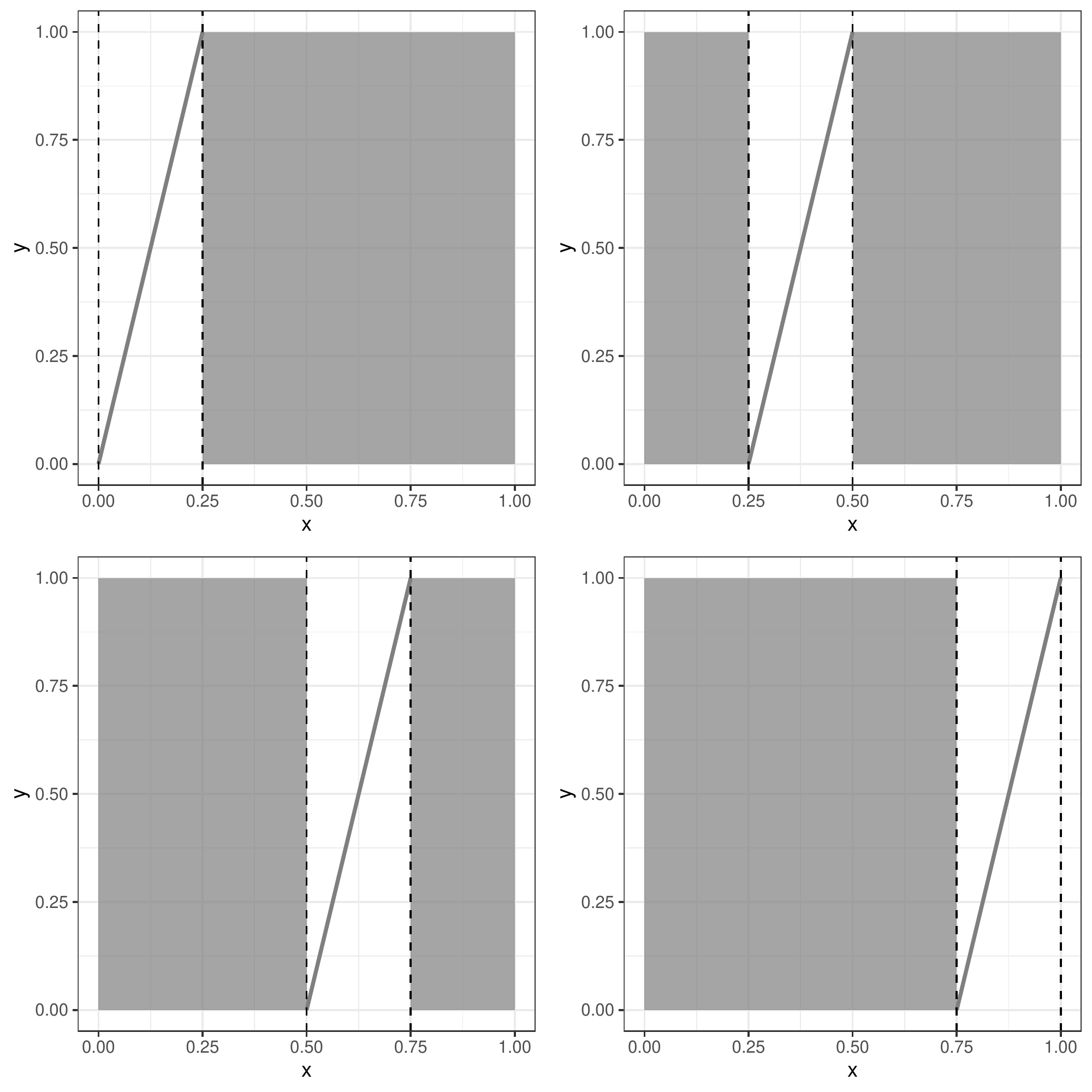}
			\caption{Supports of the copulas $C_3,C_4,C_5,C_6$ considered in Example \ref{ex:d1.not.wcc}.}
			\label{d1.not.wcc}
			\vspace*{-3mm}
		\end{figure}

  \begin{Rem} \label{delta.convergence.rem1}
	The preceding examples imply that weak conditional convergence does not imply $\partial$-convergence 
	(Example \ref{wcc.symm.ex}) nor vice versa (Example \ref{ex:d1.not.wcc}).
	\end{Rem}
	
	\section{Archimedean copulas}\label{section:archimedean}
	Recall that a generator of a bivariate Archimedean copula (see \cite{Nelsen}) is a convex and strictly decreasing function $\varphi:[0,1]\rightarrow [0,\infty]$ with $\varphi(1) = 0$. 	
	Every generator induces a copula $C$ via
	\begin{align}
	C(x,y) = \varphi^{-}(\varphi(x)+\varphi(y)), \ \ \ x,y\in [0,1]
	\end{align}
	where $\varphi^{-}:[0,\infty] \rightarrow [0,1]$ denotes the pseudoinverse of $\varphi$ defined by
	\begin{align}
	\varphi^{-}(x) := \begin{cases}
	\varphi^{-1}(x) & \text{if } x\in[0,\varphi(0+)) \\
	0 & \text{if } x\geq \varphi(0+)
	\end{cases}
	\end{align}
	where $\varphi(x\pm) := \lim_{t\to x^\pm}\varphi(t)$ denotes the respective one-sided limit. We refer to $C$ as the  Archimedean copula induced by $\varphi$ and call $C$ \emph{strict} if $\varphi(0+)=\infty$ and \emph{non-strict} otherwise.
	In what follows $\mathcal{C}_{ar}$ will denote the family of all bivariate Archimedean copulas. 

   Since $\varphi$ is convex obviously $\varphi(0) \geq \varphi(0+)$ holds.
	Defining the right-continuous version $\psi$ of $\varphi$ by
	$$
		\psi(t) := 
		\begin{cases}
		  \varphi(0+) & \text{if } t=0 
			\\
		  \varphi(t)  & \text{if } t \in (0,1]
		\end{cases}
	$$
	it is straightforward to verify that $\psi$ and $\varphi$ generate the same copula.
	In other words, the value of $\varphi$ at $0$ is irrelevant and we may, without loss of generality,  
	from now on assume that all generators are right-continuous at $0$.  
	Additionally, since for every generator $\varphi$ and every constant $a>0$ we have that $a \cdot \varphi$ generates the 
	same copula we will from now on also assume (without explicit reference) that the generator is normalized in the 
	sense that $\varphi(\frac{1}{2}) = 1$ holds. 
		
  Following \cite{p21,Nelsen} we define the $t$-level set $L_t$ of the Archimedean copula $C$ by 
	$$
	L_t := \lbrace (x,y)\in[0,1]^2: C(x,y) = t \rbrace
	$$
	and the $t$-level function $f^t: [t,1]\rightarrow[0,1]$ by $f^t(x) := \varphi^{-1}(\varphi(t)-\varphi(x))$ implying that
	\begin{align*}
	\text{graph}(f^t) = \lbrace (x,f^t(x)): x\in[t,1] \rbrace = \lbrace (x,y)\in[0,1]^2: C(x,y) = t \rbrace = L_t
	\end{align*}
	holds for every $t\in (0,1]$.
	
	For every generator $\varphi:[0,1] \rightarrow [0,\infty]$ we will let $D^+\varphi(x)$ ($D^-\varphi(x)$) denote the right-hand (left-hand) derivative of $\varphi$ at $x \in (0,1)$. Convexity of $\varphi$ implies that  
	$D^+\varphi(x)=D^-\varphi(x)$ holds for all but at most countably many $x \in (0,1)$, i.e. $\varphi$ is differentiable outside a countable subset of $(0,1)$, and that $D^+\varphi$ is non-decreasing and right-continuous (see \cite{kannan, pollard}). 
	In the sequel we will let $\text{Cont}(D^+\varphi) \subseteq (0,1)$ denote the set of all continuity points of $D^+\varphi$ in $(0,1)$ (by definition, $0$ and $1$ are not contained in $\text{Cont}(D^+\varphi))$ and make use of the fact that $[0,1] \setminus \text{Cont}(D^+ \varphi)$ is at most countably infinite and has Lebesgue measure $0$. Setting $D^+\varphi(0)=-\infty$ in case of strict $\varphi$ as well as $D^+\varphi(1)=0$ (for strict and non-strict $\varphi$) allows to view 
	$D^+\varphi$ as non-decreasing and right-continuous function on the full unit interval $[0,1]$. 
	Additionally (again see \cite{kannan,pollard})  we have $D^-\varphi(x)=D^+\varphi(x-)$ for every $x \in (0,1)$.
	
	If $\varphi$ is strict then according to \cite{p21} 
	\begin{align}\label{markov.strict}
	K_C(x,[0,y]) = \begin{cases}
	\frac{D^+\varphi(x)}{D^+\varphi(C(x,y))} & \text{if } x\in(0,1)\\
	1 & \text{if } x\in\lbrace 0,1\rbrace
	\end{cases} 
	\end{align}
	is (a version of) the Markov kernel of $C$, if $\varphi$ is non-strict, then 
	\begin{align}\label{markov.non.strict}
	K_C(x,[0,y]) = \begin{cases}
	0 & \text{if } x\in(0,1), y < f^0(x) \\
	\frac{D^+\varphi(x)}{D^+\varphi(C(x,y))} & \text{if } x\in(0,1), y\geq f^0(x)\\
	1 & \text{if } x\in\lbrace 0,1\rbrace
	\end{cases} 
	\end{align}
	is a (version of a) Markov kernel of $C$. 	
	Recall that for every Archimedean copula $C$ with generator $\varphi$ the Kendall distribution function is given by (see, e.g., \cite{GenestInference})
	\begin{align}\label{formulat.kendall.arch}
	F^\text{Kendall}(x) = x - \frac{\varphi(x)}{D^+\varphi(x)}.
	\end{align}
	
	We now prove in several steps that in $\mathcal{C}_{ar}$ weak conditional convergence and pointwise convergence coincide. 
	The following theorem serves as starting point for this result. Up to a slight modification this intermediate theorem was already established by Charpentier and Segers in \cite{convArch}, however, the slight modification will turn out to be crucial in the sequel.

	\begin{Th}\label{T1}
		Let $C, C_1, C_2,\ldots$ be Archimedean copulas with gene\-rators $\varphi, \varphi_1, \varphi_2,\ldots$, respectively. Then the following conditions are equivalent:
		\begin{enumerate}
			\item[(a)] $\lim\limits_{n\rightarrow\infty} C_n(x,y) = C(x,y)$ for all $x,y\in[0,1]$,
			\item[(b)] $\lim\limits_{n\rightarrow\infty} F_n^{\emph{Kendall}}(x) = F^{\emph{Kendall}}(x)$ for all $x\in$ \emph{Cont}$(D^+\varphi)$,
			\item[(c)] $\lim\limits_{n\rightarrow\infty}\varphi_n(x) = \varphi(x)$ for all $x\in(0,1]$,
			\item[(d)] $\lim\limits_{n\rightarrow\infty}D^+\varphi_n(x) = D^+\varphi(x)$ for all $x\in$ \emph{Cont}$(D^+\varphi)$.		 
		\end{enumerate}
	\end{Th}
	\begin{proof}
	  The equivalence of (a) and (b) can be found in \cite[Proposition 2]{convArch} and 
		the equivalence of (a) and (c) is contained in \cite[Theorem 8.14]{tnorms}
		where the authors prove equivalence of pointwise convergence of the sequence of multiplicative generators and 
		the pointwise convergence of the induced continuous Archimedean $t$-norms which readily translates to the copula setting.
		The fact that (c) implies (d) is a direct consequence of the convexity of the generators, see \cite{rockafellar}. 	
		
		Finally, suppose that (d) holds and consider $x \in \textrm{Cont}(D^+\varphi)$. 
		For every $\varepsilon >0$ we can find an index $n_0 \in \mathbb{N}$ such that for all $n\geq n_0$ we have 
		$D^+\varphi(x) - \varepsilon < D^+\varphi_n(x)$. Using monotonicity of $D^+\varphi_n$ we get
		$$
		D^+\varphi(x) - \varepsilon \leq D^+\varphi_n(x) \leq D^+\varphi_n(t) \leq 0 
		$$
		for $n \geq n_0$ and every $t \in [x,1]$.
		Having this, Dominated convergence yields
		\begin{align*}
		-\varphi(y) 
		&= \int\limits_{[y,1]} D^+\varphi(t) \, \mathrm{d}t 
		 = \int\limits_{[y,1]} \lim_{n \rightarrow \infty} D^+\varphi_n(t) \,  \mathrm{d}t   
		= \lim_{n \rightarrow \infty} \int\limits_{[y,1]} D^+\varphi_n(t) \,  \mathrm{d}t 
		= -\lim_{n \rightarrow \infty} \varphi_n(y)
		\end{align*} 
		for every $y \in [x,1]$. 		
		Considering that $\textrm{Cont}(D^+\varphi)$ is dense in $(0,1]$ condition (c) now follows immediately.
	\end{proof}

	The afore-mentioned modification of the result in \cite{convArch} is that it may happen that
	a sequence of Archimedean copulas $(C_n)_{n\in\mathbb{N}}$ converges to an Archimedean copula $C$
	although the corresponding generators $(\varphi_n)_{n\in\mathbb{N}}$ do not converge to $\varphi$ in the point $0$.
		
	We now state the main result of this section saying 
	that pointwise convergence and weak conditional convergence coincide in $\mathcal{C}_{ar}$: 

	\begin{Th}\label{result}
		Let $C, C_1,C_2,\ldots$ be Archimedean copulas with gene\-rators $\varphi, \varphi_1,\varphi_2,\ldots$, respectively. Then the following assertions are equivalent:
		\begin{enumerate}
			\item[(a)] $\lim\limits_{n\rightarrow\infty} C_n(x,y) = C(x,y)$ for all $x,y\in[0,1]$,
			\item[(b)] $\lim\limits_{n\rightarrow\infty} F_n^{\emph{Kendall}}(x) = F^{\emph{Kendall}}(x)$ for all $x\in$ \emph{Cont}$(D^+\varphi)$,
			\item[(c)] $\lim\limits_{n\rightarrow\infty}\varphi_n(x) = \varphi(x)$ for all $x\in(0,1]$,
			\item[(d)] $\lim\limits_{n\rightarrow\infty}D^+\varphi_n(x) = D^+\varphi(x)$ for all $x\in$ \emph{Cont}$(D^+\varphi)$,
			\item[(e)] $\lim\limits_{n\rightarrow\infty}D_1(C_n,C) = 0$,
			\item[(f)] $C_n \xrightarrow{wcc} C$ for $n \rightarrow \infty$.
		\end{enumerate}
	\end{Th}
	
	\begin{Rem}
	Since Archimedean copulas are symmetric Theorem \ref{result} implies that 
	within the class of all Archimedean copulas weak conditional convergence and 
	$\partial$-convergence are equivalent (compare with Remark \ref{delta.convergence.rem1}).
	\end{Rem}
	
	We are now going to prove Theorem \ref{result} 
	by showing that each of the conditions (a) to (d) from Theorem \ref{T1} implies weak conditional convergence.	
	Doing so we will work with level curves and distinguish two cases concerning the $0$-level curve $f^0$ of the limit copula.
	For the first case, we present two different proofs since they use different ideas - the first one 
	builds upon convexity of the generators and direct consequences to the sequence of derivatives, the second one 
	uses some additional information about the behaviour of the corresponding sequence of level curves 
	as described in the next lemma: \pagebreak
	
	\begin{Lemma}\label{direction}
		Let $C,C_1, C_2,\ldots$ be non-strict Archimedean copulas with generators 
		\linebreak $\varphi, \varphi_1, \varphi_2, \ldots$, respectively. If $(\varphi_n)_{n\in\mathbb{N}}$ converges pointwise to $\varphi$ on $(0,1]$ then
		the following two assertions hold:
		\begin{enumerate}
			\item[(i)] $\liminf\limits_{n\rightarrow \infty} \varphi_n(0) \geq \varphi(0)$,
			\item[(ii)] $f^t(x) \geq \limsup\limits_{n \rightarrow \infty} f_n^t(x)$ for every $t\in[0,1)$ and every $x\in [t,1]$.
		\end{enumerate} 
	\end{Lemma}
	\begin{proof}
		To prove assertion (i) we proceed as follows: 
		Considering that for every ge\-ne\-rator $\psi$ and every $z\in (0,1)$ by convexity we have 
		$\psi(0)\geq \psi(z) - z\cdot D^+\psi(z)$ and applying Theorem \ref{T1} to the case $z \in \textrm{Cont}(D^+\varphi)$
		yields
		\begin{align*}
		\liminf_{n\rightarrow \infty} \varphi_n(0) \geq \liminf_{n\rightarrow \infty} 
		\left( \varphi_n(z) - z\cdot D^+\varphi_n(z)\right) = \varphi(z) - z\cdot D^+\varphi(z) > \varphi(z).
		\end{align*} 
		Since for every $\varepsilon>0$ we can choose $z \in \textrm{Cont}(D^+\varphi)$ in such a way that 
		$\varphi(z)> \varphi(0)-\varepsilon$	holds assertion (i) now follows.
		To prove assertion (ii) fix $t\in[0,1)$ and consider $y> f^t(x)$. In this case we have 
		$\varphi(x)+\varphi(y) < \varphi(t)$ implying that there exists an index $n_0 \in \mathbb{N}$ such that 
		$\varphi_n(x)+\varphi_n(y) < \varphi_n(t)$, hence $y> f_n^t(x)$, holds for every $n\geq n_0$. 
		It follows that $y\geq \limsup_{n \rightarrow \infty} f_n^t(x)$ from which the assertion follows immediately since 
		$y> f^t(x)$ was arbitrary. 
	\end{proof}

	We now use the previous result to show level curve convergence:
	
	\begin{Lemma}\label{levelConvergence}
		Let $C,C_1, C_2, \ldots$ be Archimedean copulas with gene\-rators $\varphi, \varphi_1, \varphi_2, \ldots$ converging pointwise on $(0,1]$. Then for every $t>0$ the $t$-level curves converge pointwise, i.e., 
		\begin{align}\label{eq:cont.levelt}
		\lim_{n \rightarrow \infty} f^t_n(x) = f^t(x)
		\end{align}
		holds for all $x\in[t,1]$. If, in addition, $\lim_{n \rightarrow \infty} \varphi_n(0)=\varphi(0)$ holds then 
		eq. (\ref{eq:cont.levelt}) is also true for $t=0$ and $x \in [0,1]$.
	\end{Lemma}
	\begin{proof}
		Suppose that $t>0$. As a by-product of \cite[Theorem 8.14]{tnorms} we obtain uniform convergence of the sequence 
		$(\varphi_n^{-})_{n \in \mathbb{N}}$ to $\varphi^-$ on each compact interval of the form $[0,s]$ with $s\in[0,\infty)$.
		Fix $x \in [t,1]$, set $s:=2\sup_{n \in \mathbb{N}} \varphi_n(t) < \infty$ and define 
		$z_n := \varphi_n(t) - \varphi_n(x)$. Then $(z_n)_{n \in \mathbb{N}}$ converges to $z:=\varphi(t) - \varphi(x)$ for 
		$n\to\infty$ and, using the afore-mentioned uniform convergence of $(\varphi_n^{-})_{n \in \mathbb{N}}$ 
		on $[0,s]$, the equality $\lim_{n \rightarrow \infty} f^t_n(x) = f^t(x)$ follows.\\
		Notice that the second assertion is trivial for strict $\varphi$, so it remains to prove the assertion for $\varphi(0)<\infty$. Fix $x \in (0,1)$. If $f^0(x)=0$ then the result follows directly from Lemma \ref{direction} part (ii). 
		Suppose therefore that $f^0(x)>0$. Then for every 
		$y \in (0,f^0(x))$ we have $\varphi(x) + \varphi(y)> \varphi(0)$ and we can find an index $n_0 \in \mathbb{N}$ 
		such that $\varphi_n(x) + \varphi_n(y) > \varphi_n(0)$, hence $y<f^0_n(x)$, holds for every $n\geq n_0$. 
		As direct consequence we get $f^0(x) \leq \liminf_{n\rightarrow \infty} f_n^0(x)$, which in combination 
		with Lemma \ref{direction} assertion (ii) yields
		\begin{align*}
		f^0(x) \leq \liminf_{n\rightarrow \infty} f_n^0(x) \leq \limsup_{n\rightarrow \infty} f_n^0(x) \leq f^0(x).
		\end{align*}
		This completes the proof.	
	\end{proof}
	
	Recall that for univariate distribution functions $F,F_1,F_2,\ldots$ weak convergence of $(F_n)_{n \in \mathbb{N}}$ to $F$ is equivalent to pointwise convergence on a dense subset (see, e.g., \cite{billingsley}). In the following two lemmata we prove convergence on a dense set above and below the zero level curve $f^0$ of the limit copula $C$. Notice that the first lemma is sufficient within the family of strict Archimedean copulas since in this case $f^0(x)=0$ for every $x \in (0,1]$. 
	
	\begin{Lemma}\label{main}
		Let $C,C_1, C_2, \ldots$ be Archimedean copulas with generators $\varphi, \varphi_1, \varphi_2, \ldots$ and assume that one of the conditions of Theorem \ref{T1} holds. Then there exists a set $\Lambda \in \mathcal{B}([0,1])$ fulfilling $\lambda(\Lambda)=1$ such that for every $x \in \Lambda$ we have that 
		\begin{align}
		\lim_{n \rightarrow \infty}K_{C_n}(x,[0,y]) = K_{C}(x,[0,y])
		\end{align}
		holds for every $y \in U_x \subseteq [f^0(x),1]$, where $U_x$ is dense in $[f^0(x),1]$.
	\end{Lemma}
	\begin{proof}[Proof (1)]
		Setting $\Lambda:=\textrm{Cont}(D^+\varphi)$ we obviously have $\lambda(\Lambda)=1$. We are going to prove the even
		stronger property that for every $x \in \Lambda$
		the identity
		\begin{align}
		\lim_{n \rightarrow \infty} |K_{C_n}(x,[0,y]) - K_C(x,[0,y])| &= \lim_{n \rightarrow \infty} \bigg| \frac{D^+\varphi_n(x)}{D^+\varphi_n(C_n(x,y))} - \frac{D^+\varphi(x)}{D^+\varphi(C(x,y))} \bigg| \\ &= 0 \nonumber
		\end{align}
		holds for $\lambda$-almost all $y \in [f^0(x),1]$. 
		First of all notice that the set
		\begin{align*}
		U_x := \lbrace y \in [f^0(x),1]: C(x,y) \in\text{ Cont}(D^+\varphi) \rbrace
		\end{align*}
		is of full measure in $[f^0(x),1]$. 
		In fact, for $x \in (0,1)$ the function $h_x:[f^0(x),1] \rightarrow [0,x]$, defined by $h_x(y) := C(x,y)$ is an increasing homeomorphism (see \cite{Nelsen}) and therefore the set
			$h_x^{-1}(\text{Cont}(D^+\varphi)^c)$ is at most countably infinite.
		 Convexity of the generators implies that the sequence $(D^+\varphi_n)_{n\in\mathbb{N}}$ of 
		 derivatives converges continuously to 
		$D^+\varphi$ on Cont$(D^+\varphi)$ (see \cite{rockafellar}), hence we obtain
		\begin{align*}
		\lim_{n \rightarrow \infty} D^+\varphi_n(C_n(x,y)) = D^+\varphi(C(x,y))
		\end{align*}
		from which the desired property follows.
	\end{proof}
	\begin{proof}[Proof (2)]
		First of all notice that for $t,x \in \textrm{Cont}(D^+\varphi)$ we have
		\begin{align}\label{eq:temp.kernel.expression}
		K_{C_n}(x,[0,f_n^t(x)]) = \frac{D^+\varphi_n(x)}{D^+\varphi_n(C(x,f^t_n(x)))} = 	\frac{D^+\varphi_n(x)}{D^+\varphi_n(t)}
		\end{align}
		and the right-hand side converges to $K_C(x,[0,f^t(x)])$ by Theorem \ref{T1}. Exploiting this fact we consider 
		$x \in \textrm{Cont}(D^+\varphi)$ and proceed as follows:
		Fix $\varepsilon>0$ again and let $y \in  [f^0(x),1]$ denote a continuity point of the map $v\mapsto K_C(x,[0,v])$. 
		Furthermore choose $t,s \in \textrm{Cont}(D^+\varphi)$ with $t<s$ in such a way that $y \in (f^t(x), f^s(x))$ and 
		$K_C(x,[f^t(x),f^s(x)]) < \varepsilon$ holds. According to Lemma \ref{levelConvergence} there exists some index $n_0 \in \mathbb{N}$ such that for all $n\geq n_0$ we have $y < f_n^s(x)$, which using eq. (\ref{eq:cont.levelt}) implies 
		\begin{align*}
		K_{C_n}(x,[0,y]) \leq K_{C_n}(x,[0,f_n^s(x)]) \xrightarrow{n\to\infty} K_C(x,[0,f^s(x)]) \leq K_C(x,[0,y]) +\varepsilon.
		\end{align*}
		As a direct consequence we get
		\begin{align*}
		\limsup_{n\rightarrow \infty}& \,K_{C_n}(x,[0,y]) \leq K_C(x,[0,y]) + \varepsilon.
		\end{align*}
		Replacing $s$ by $t$ and proceeding analogously yields 
		\begin{align*}
		\liminf_{n\rightarrow \infty} & \, K_{C_n}(x,[0,y]) \geq K_C(x,[0,y]) - \varepsilon,
		\end{align*}
		which completes the proof.
	\end{proof}
	
	As second step we consider the case $0< y < f^0(x)$ (implying that $C$ is non-strict). 
	The following simple lemma will be crucial for the proof of Lemma \ref{main2}:
	\begin{Lemma}\label{subsequences}
			Suppose that $(X,d_1)$ is a compact metric space and that $(Y,d_2)$ is another (not necessarily compact) metric space. Furthermore let $f:X\rightarrow Y$ be an arbitrary function and $(x_n)_{n\in\mathbb{N}}$ a sequence in $(X,d_1)$. 
			If there exists some $y \in Y$ such that for every convergent subsequence $(x_{n_j})_{j\in\mathbb{N}}$ we have $\lim_{j \rightarrow \infty} d_2(f(x_{n_j}),y)=0$ then $\lim_{n \rightarrow \infty}d_2(f(x_n),y)=0$ follows. 
	\end{Lemma}
	\begin{proof}
			Suppose that the assumptions of the lemma are fulfilled but the sequence $(y_n)_{n \in \mathbb{N}}$ with
			$y_n := f(x_n)$ does not converge to $y$ for $n\to\infty$. 
			In this case there exists some $\varepsilon>0$ and a subsequence $(y_{n_j})_{j \in \mathbb{N}}$ with
			$d_2(y_{n_j}, y) \geq \varepsilon$ for every $j\in\mathbb{N}$. 
			Compactness of $(X,d_1)$ implies the existence of a subsequence $(x_{n_{j_l}})_{l\in\mathbb{N}}$ 
			of $(x_{n_j})_{j\in\mathbb{N}}$, and, by assumption, this sequence fulfills  
			$$
			\lim_{l \rightarrow \infty} d_2(y_{n_{j_l}},y)=\lim_{l \rightarrow \infty} d_2(f(x_{n_{j_l}}),y)=0,
			$$
			a contradiction. 
	\end{proof}
	
	To simplify notation we say that 
	$\lim_{n\to\infty} \varphi_n(0) = \infty $ if, and only if for every $N \in \mathbb{N}$ there exists some 
	index $n_0\in\mathbb{N}$ such that for all $n\geq n_0$ we have $\varphi_n(0) > N$.
   Lemma \ref{direction} part (i) together with Lemma \ref{subsequences} allow us to distinguish the following three types of convergent subsequences $(\varphi_{n_j})_{j \in \mathbb{N}}$ of $(\varphi_{n})_{n \in \mathbb{N}}$: (a) $\lim_{j \rightarrow \infty} \varphi_{n_j}(0)=\varphi(0)$, 
	(b) $\lim_{j \rightarrow \infty}\varphi_{n_j}(0) =\infty$ or 
	(c) $\lim_{j \rightarrow \infty}\varphi_{n_j}(0) = \alpha \in (\varphi(0),\infty)$.  
	
	\begin{Lemma}\label{main2}
		Let $C,C_1, C_2, \ldots$ be Archimedean copulas with generators $\varphi, \varphi_1, \varphi_2, \ldots$ and assume that one of the conditions of Theorem \ref{T1} holds. Then there exists a set $\Lambda \in \mathcal{B}([0,1])$ fulfilling $\lambda(\Lambda)=1$ such that for every $x \in \Lambda$ we have that 
		\begin{align}
		\lim_{n \rightarrow \infty}K_{C_n}(x,[0,y]) = 0 = K_{C}(x,[0,y])
		\end{align}	holds for every $y < f^0(x)$.
	\end{Lemma}
	\begin{proof}
		As in the previous proof we set $\Lambda=\textrm{Cont}(D^+\varphi)$. 
		Fix $x\in$ Cont$(D^+\varphi)$ and $ y \in (0, f^0(x))$ and distinguish the following two different situations: \\
		(a) Suppose that $(\varphi_{n_j})_{j \in \mathbb{N}}$ is a subsequence of 
		$(\varphi_{n})_{n \in \mathbb{N}}$ fulfilling $\lim_{j \rightarrow \infty} \varphi_{n_j}(0) = \varphi(0)$. 
		According to Lemma \ref{levelConvergence} we have $\lim_{j \rightarrow \infty}f^0_{n_j}(x)=f^0(x)$, so there exists an index $j_0 \in \mathbb{N}$ such that $y < f^0_{n_j}(x)$, hence $K_{C_{n_j}}(x,[0,y])=0=K_C(x,[0,y])$, 
		holds for every $j \geq j_0$, from which $\lim_{j \rightarrow \infty} K_{C_{n_j}}(x,[0,y])=K_C(x,[0,y])$ follows immediately. \\	
		(b) \& (c) Suppose that $(\varphi_{n_j})_{j \in \mathbb{N}}$ is a subsequence of 
		$(\varphi_{n})_{n \in \mathbb{N}}$ fulfilling $\lim_{j \rightarrow \infty} \varphi_{n_j}(0) = \infty$ or $\lim_{j \rightarrow \infty} \varphi_{n_j}(0) = :\alpha \in (\varphi(0),\infty)$. 
		Choose $\varepsilon>0$ in such a way that $y\leq f^0(x)-\varepsilon$ holds and define the set $M_\varepsilon$ (see Figure \ref{Meps}) by 
		$$
		M_\varepsilon = \lbrace (a,b)\in[0,1]^2: b \leq f^0(a)-\varepsilon \rbrace.
		$$
		\begin{figure}[h!]
			\centering
			\includegraphics[width=0.75\textwidth]{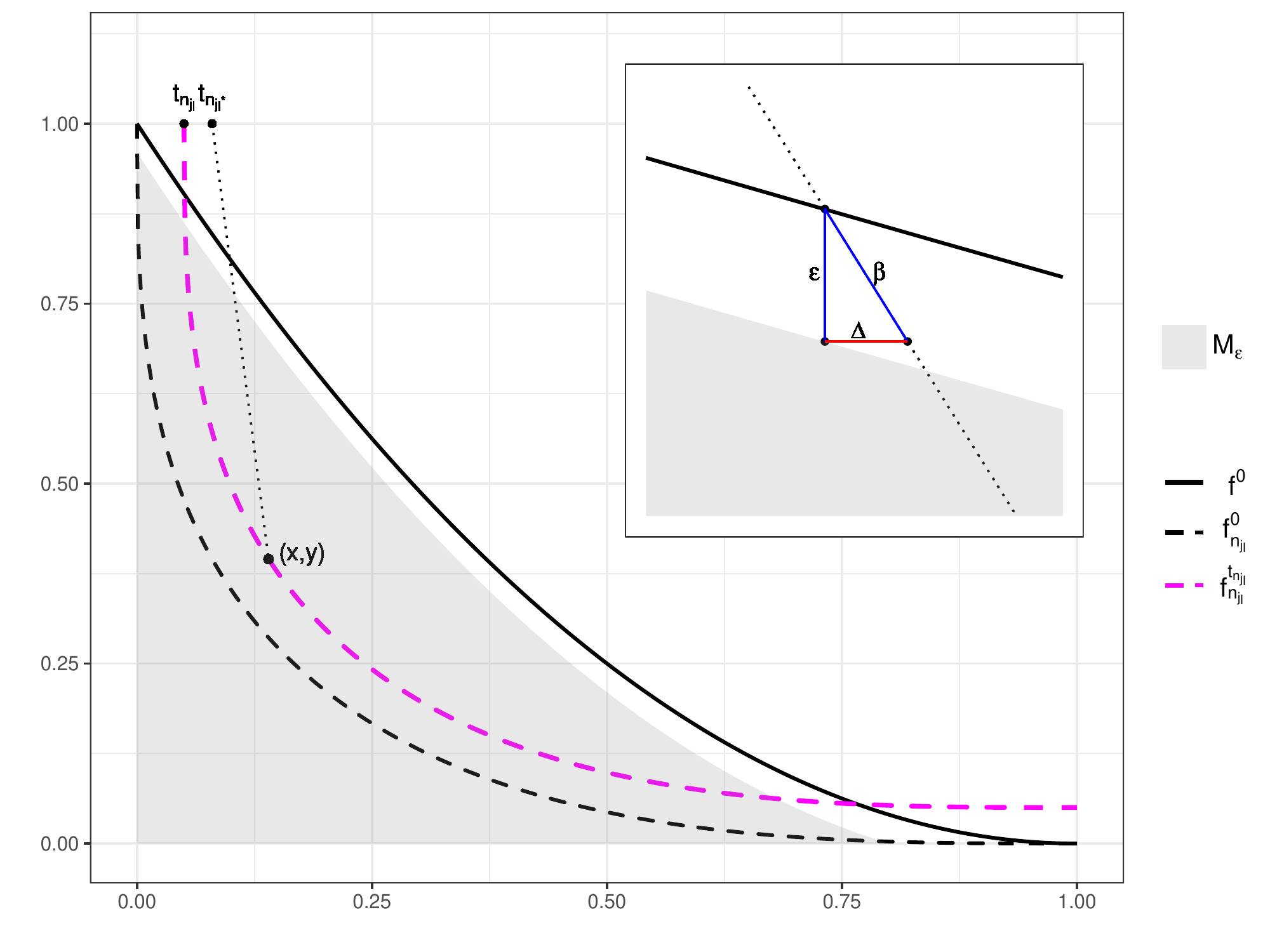}
			\caption{The $\mu_C$-null set $M_\varepsilon$, the level curves $f^{t_{n_{j_l}}}_{n_{j_l}}, f^0_{n_{j_l}}, f^0$ considered in the proof of Lemma \ref{main2} and a zoomed-in illustration of $\Delta$, where $\beta = \frac{1-y}{x-t_{n_{j_{l^*}}}}$.}\label{Meps}
		\end{figure}
		Then $\mu_C(M_\varepsilon) = 0$ and $M_\varepsilon$ is a $\mu_C$-continuity set, so applying Portmanteau's theorem (see \cite{billingsley}) yields 
		\begin{align}\label{proof:Me}
		\lim_{j \rightarrow \infty}\mu_{C_{n_j}}(M_\varepsilon) = 0.
		\end{align}
		Assume now that there exists some $\delta>0$ such that $K_{C_{n_j}}(x,[0,y]) \geq \delta>0$ would hold for infinitely many $j \in \mathbb{N}$ and denote the corresponding subsequence by $(C_{n_{j_l}})_{l \in \mathbb{N}}$. 
		It follows from eq. (\ref{markov.non.strict}) that $y \geq f^0_{n_{j_l}}(x)$ holds for every $l \in \mathbb{N}$. 
		Set $t_{n_{j_l}} := C_{n_{j_l}}(x,y)$ for every $l \in \mathbb{N}$ and let $l^*$ denote the smallest index fulfilling 
		that $t_{n_{j_l}} < x$ holds for all $l \geq l^*$.
		Using the fact that for every Archimedean copula $A$ with generator $\psi$ and for every $t \in [0,1)$ 
		the mapping	
		\begin{align*}
		x \mapsto	K_{A}(x,[0,f^t(x)]) = \frac{D^+\psi(x)}{D^+\psi(t)}
		\end{align*}
		is decreasing in $x$ it follows that for every $l \in \mathbb{N}$ we have 
		$$
		K_{C_{n_{j_l}}}(u, [0,f_{n_{j_l}}^{t_{n_{j_l}}}(u)]) \geq  K_{C_{n_{j_l}}}(x, [0,\underbrace{f_{n_{j_l}}^{t_{n_{j_l}}}(x)}_{=y}])  \geq  \delta >0
		$$ 
		for every $u \in (0,x]$. The proof idea now is to use this monotonicity in combination with convexity of the level curves (see \cite{Nelsen}) to construct a contradiction to eq. (\ref{proof:Me}): In fact, convexity implies that 
		the graph of each $f_{n_{j_l}}^{t_{n_{j_l}}}$ restricted to $[t_{n_{j_{l^*}}},x]$ lies below the straight line connecting the points $(x,y)$ and $(t_{n_{j_{l^*}}},1)$ (again see Figure \ref{Meps}). Hence, defining $\Delta>0$ by
		\begin{align}
		\Delta = \frac{\varepsilon}{\frac{1-y}{x-t_{n_{j_{l^*}}}}} = \varepsilon \, \frac{x-t_{n_{j_{l^*}}}}{1-y\,\,\,}
		\end{align}
		it follows that for every $u \in [x-\Delta,x]$ we have $(u,f_{n_{j_l}}^{t_{n_{j_l}}}(u)) \in M_\varepsilon$. 
		As direct consequence it follows that
		\begin{align*}
		\mu_{C_{n_{j_{l}}}}(M_\varepsilon) \geq \int_{[x-\Delta, x]} 
		K_{C_{n_{j_{l}}}}(u,[0,f_{n_{j_{l}}}^{t_{n_{j_{l}}}}(u)]) \ \mathrm{d}\lambda(u) \geq \delta \cdot \Delta > 0
		\end{align*}
		holds for every $l > l^*$ which contradicts eq. (\ref{proof:Me}). 
		Altogether in case (b) \& (c) we have also shown now that  $\lim_{j \rightarrow \infty} K_{C_{n_j}}(x,[0,y]) = 0 = K_C(x,[0,y])$ holds. \\
		Taking (a) and (b) \& (c) together we have proved that for each convergent subsequence $(\varphi_{n_j}(0))_{j \in \mathbb{N}}$
		of $(\varphi_n(0))_{n \in \mathbb{N}}$ we have 
		$$
		\lim_{j \rightarrow \infty}K_{C_{n_j}}(x,[0,y]) = 0. 
		$$
		The result now follows from Lemma \ref{subsequences}. 
	\end{proof}
	
	Considering that weak conditional convergence of the copulas implies convergence in $D_1$ the proof of this section's main result Theorem \ref{result} is complete. In Section \ref{est} we will use the following interesting consequence: 	
	\begin{Cor}\label{cor:main1}
		Let $C, C_1,C_2,\ldots$ be Archimedean copulas with gene\-rators $\varphi, \varphi_1,\varphi_2,\ldots$, respectively and suppose that $(\varphi_n)_{n \in \mathbb{N}}$ converges to $\varphi$ on $(0,1]$. Then the following identities holds:
		\begin{equation}
		\lim_{n \rightarrow \infty} \zeta_1(C_n)= \zeta_1(C),\quad \lim_{n \rightarrow \infty} r(C_n)= r(C)
		\end{equation}
		In other words: Within $\mathcal{C}_{ar}$ both $\zeta_1$ and $r$ are continuous w.r.t. pointwise convergence of the generators on $(0,1]$.
	\end{Cor} 
	
%

	We conclude this section by recalling the fact that the class of Archimedean copulas is not closed 
	w.r.t. uniform convergence, i.e., the limit of a sequence of Archimedean copulas may fail to be Archimedean, see \cite{convArch} (however, we necessarily have associativity). 
	As easily verified, the same is true if we consider weak conditional convergence or convergence w.r.t. $D_1$.
	
	
	\section{Extreme Value copulas}
	We are now going to prove a result similar to Theorem \ref{result} for bivariate Extreme Value copulas. 
	Remember that $C\in\mathcal{C}$ is called bivariate \emph{Extreme Value copula} if one of the following three equivalent conditions is fulfilled (see \cite{deHaan, Principles, GS2, Pickands}):
	\begin{enumerate}
		\item[(a)] There is a copula $B\in\mathcal{C}$ such that for all $x,y\in [0,1]$ we have
		\begin{align}
		C(x,y) = \lim_{n \rightarrow \infty} B^n(x^{\frac{1}{n}}, y^\frac{1}{n}).
		\end{align}
		\item[(b)] $C(x,y) = C^n(x^\frac{1}{n}, y^\frac{1}{n})$ holds for all $n\in\mathbb{N}$ and all $x,y\in[0,1]$.
		\item[(c)] There exists a convex map $A:[0,1]\to[0,1]$ satisfying $A(0) = A(1) = 1$ and $\max(1-x,x)\leq A(x) \leq 1$ for all $x\in [0,1]$ such that for all $x,y\in(0,1)$ the copula $C$ can be expressed in terms of $A$ as
		\begin{align}
		C(x,y) = C_A(x,y):=(xy)^{A\big(\frac{\ln(x)}{\ln(xy)}\big)}.
		\end{align}
	\end{enumerate}
In the following we will let $\mathcal{C}_{ev}$ denote the class of all bivariate Extreme Value copulas, $\mathcal{A}$ the family of all \emph{Pickands dependence functions}, i.e., the family of all functions $A$ fulfilling assertion (c). 
	Using either max-stability or Arzela-Ascoli theorem \cite{rudinrc} it is straightforward to verify that 
	$\mathcal{C}_{ev}$ is a compact subset of $(\mathcal{C},d_\infty)$. Furthermore, letting $\Vert \cdot \Vert_\infty$ denote the uniform norm on $\mathcal{A}$, obviously the mapping 
	$\Phi: (\mathcal{A},\Vert \cdot \Vert_\infty) \rightarrow (\mathcal{C}_{ev},d_\infty)$, defined by
	$\Phi(A)=C_A$, is continuous and it is straightforward to verify that a sequence of Extreme Value copulas $(C_{A_n})_{n \in \mathbb{N}}$ converges pointwise (hence uniformly) to an Extreme Value copula $C_A$ if, and only if 
	$(A_n)_{n\in\mathbb{N}}$ converges uniformly to $A$. 
	
	Following \cite{p23} we will let $D^+A$ denote the right-hand derivative of the Pickands dependence function $A$ on $[0,1)$ 
	and $D^-A$ the left-hand derivative on $(0,1]$. Furthermore, convexity implies that $D^-A(x) = D^+A(x)$ holds for all but at most countably infinitely many $x \in (0,1)$. In the sequel we will let $\text{Cont}(D^+A)$ denote 
	the set of all continuity points of $D^+A$ in $(0,1)$. 
	Setting $D^+A(1): = D^-A(1)$ we can view $D^+A$ as a function on the whole unit interval that attains values in $[-1,1]$. 
	Furthermore $D^+A: [0,1]\to [-1,1]$ is a non-decreasing, right-continuous function and it is straightforward to verify that 
	$\mathcal{A}$ can be identified with $\mathcal{D}_\mathcal{A}$, defined by
	\begin{equation*}
	\mathcal{D}_\mathcal{A}=\bigg\{f: [0,1] \rightarrow [-1,1]: f \textrm{ non-decreasing,  
		right-continuous}, \int_{[0,1]} f \ \mathrm{d}\lambda =0 \bigg\},
	\end{equation*}
	in the sense that for every $A \in \mathcal{A}$ we have $D^+A \in \mathcal{D}_\mathcal{A}$ and, given 
	$f \in \mathcal{D}_\mathcal{A}$ setting $A(x):= 1+ \int_{[0,x]} f \mathrm{d}\lambda$ yields $A \in \mathcal{A}$ as well as $D^+A=f$ on $[0,1)$ (see \cite{p23}). For more information on Pickands dependence functions and the approach via right-hand derivatives we refer to \cite{Caperaa, Ghoudi}. 
	
	Returning to weak conditional convergence first notice that according to \cite{p23} 
	\begin{align} \label{kernel.evc}
	K_{C}(x,[0,y]) = \begin{cases}
	1 & \text{ if } x\in \lbrace 0, 1\rbrace \\
	C(x,y)\big[ D^+A\big(\frac{\log(x)}{\log(xy)}\big) \frac{\log(y)}{x \log(xy)} + \frac{1}{x} A\big(\frac{\log(x)}{\log(xy)}\big) \big] & \text{ if } x,y\in (0,1) \\
	y & \text{ if } x\in (0,1), y\in \lbrace 0, 1 \rbrace
	\end{cases}
	\end{align}
	is a Markov kernel of the Extreme Value copula $C$ with Pickands dependence function $A$. 
	
	We now state the main result of this section saying that in $\mathcal{C}_{ev}$ 
	pointwise convergence and weak conditional convergence are equivalent:   
	
	\begin{Th}\label{EVcorollary}
		Let $C,C_1,C_2,\ldots$ be Extreme Value copulas with Pickands dependence functions $A,A_1, A_2,\ldots$, respectively. Then the following assertions are equivalent:
		\begin{enumerate}
			\item[(a)] $\lim\limits_{n\rightarrow\infty} C_n(x,y) = C(x,y)$ for all $x,y\in[0,1]$,
			\item[(b)] $\lim\limits_{n\rightarrow\infty}A_n(x) = A(x)$ for all $x\in[0,1]$,
			\item[(c)] $\lim\limits_{n\rightarrow\infty}D^+A_n(x) = D^+A(x)$ for all $x\in$ \emph{Cont}$(D^+A)$,
			\item[(d)]  $\lim\limits_{n\rightarrow\infty}D_1(C_n,C) = 0$,
			\item[(e)] $C_n \xrightarrow{wcc} C$ for $n \rightarrow \infty$.
		\end{enumerate}
	\end{Th} 
	
	\begin{Rem}
	Since for every Extreme Value copula $C_A$ its transpose $C_{A^t} := (C_A)^t$ is an Extreme Value copula with 
	Pickands dependence function $A^t$ given by $A^t (x) = A(1-x)$
	for every $x\in[0,1]$ it follows that in the class of bivariate Extreme Value copulas the properties	
	$C_n \xrightarrow{wcc} C$ and $C_n^t \xrightarrow{wcc} C^t$ are equivalent. 
	As a consequence of Theorem \ref{EVcorollary} 
	(and in contrast to Remark \ref{delta.convergence.rem1}) 
	weak conditional convergence and $\partial$-convergence are therefore equivalent in $\mathcal{C}_{ev}$.
	\end{Rem}

	Theorem \ref{EVcorollary} is a direct consequence of the following analogue of Lemma \ref{main} and Lemma \ref{main2}
	(notice that the result implies weak conditional convergence for ANY choice of the Markov kernels):
	
	\begin{Lemma}\label{extremevalue}
		Let $C,C_1,C_2,\ldots$ be bivariate Extreme Value copulas with Pickands dependence functions $A,A_1, A_2,\ldots$, respectively. 
		Suppose that $(C_n)_{n\in\mathbb{N}}$ converges pointwise to $C$ and choose the corresponding kernels 
		according to (\ref{kernel.evc}). Then for every $x\in(0,1)$
		there exists a set $U_x \subset [0,1]$ that is dense in $[0,1]$ such that
		\begin{align}\label{temp.evc.conv}
		\lim_{n \rightarrow \infty}K_{C_n}(x,[0,y]) = K_{C}(x,[0,y])
		\end{align}
		holds for every $y \in U_x$.
	\end{Lemma}

	\begin{proof}
		If $(C_n)_{n\in\mathbb{N}}$ converges pointwise to $C$ then, as mentioned before, it follows that   
		$\lim_{n \rightarrow \infty} \Vert  A_n- A \Vert_\infty=0$ holds. 
		Thus (as in the case of Archimedean generators) convexity yields
		$$
		\lim_{n \rightarrow \infty} D^+A_n(x) = D^+A(x)
		$$ 
		for every $x\in\text{Cont}(D^+A)$. Defining $h_x:(0,1) \rightarrow (0,1)$ for every $x\in(0,1)$ by
		\begin{align}
		h_x(y)=\frac{\log(x)}{\log(xy)}
		\end{align}
		yields a strictly increasing homeomorphism of $(0,1)$. Since $(0,1) \setminus \text{Cont}(D^+A)$ is at most countably infinite $h_x^{-1}\left((0,1) \setminus \text{Cont}(D^+A)\right)$ is as well and  
		$$
		\lambda\left(h_x^{-1}(\text{Cont}(D^+A))\right) =1
		$$ 
		follows. Being a set of full measure $U_x:=h_x^{-1}(\text{Cont}(D^+A))$ is dense in $[0,1]$ and the result follows.
	\end{proof}
	
	Altogether we have proved Theorem \ref{EVcorollary} which, in turn, has the following corollary:
	\begin{Cor}\label{cor:main2}
		Let $C, C_1,C_2,\ldots$ be Extreme Value copulas with Pickands dependence functions $A, A_1,A_2,\ldots$, respectively and suppose that $(A_n)_{n \in \mathbb{N}}$ converges to $A$ on $[0,1]$. Then the following identities holds:
		\begin{equation}
		\lim_{n \rightarrow \infty} \zeta_1(C_n)= \zeta_1(C),\quad \lim_{n \rightarrow \infty} r(C_n)= r(C)
		\end{equation}
		In other words: Within $\mathcal{C}_{ev}$ both $\zeta_1$ and $r$ are continuous w.r.t. pointwise convergence of the Pickands dependence functions.
	\end{Cor} 
	
	We conclude this section with the following remark:
	\begin{Rem}
		Suppose that $H, H_1, H_2, \ldots$ are the continuous bivariate distribution functions of the pairs $(X,Y), (X_1,Y_1), (X_2,Y_2), \ldots$, with corresponding marginal distribution functions $F_X, F_{X_1}, F_{X_2},\ldots$ and $G_Y, G_{Y_1}, G_{Y_2}, \ldots$ and corresponding copulas $C, C_1, C_2,\ldots$. 
		Defining  $H_n \xrightarrow{\text{wcc}} H$ analogously to Definition \ref{def:wcc} (notice that in this case $\lambda$ is replaced by $\mathbb{P}^X$) it is straightforward to verify that 
		$H_n \xrightarrow{\text{wcc}} H$ implies $C_n \xrightarrow{\text{wcc}} C$ but not necessarily vice versa. 
		In the case that all $C,C_1,C_2,\ldots$ are Extreme Value copulas and the Pickands function $A$ of the limit copula 
		$C$ is twice differentiable, however, the reverse implication also holds.
		In the class $\mathcal{C}_{ar}$ the two concepts are equivalent too if, for instance, all generators are $3$-monotone.
	\end{Rem}

	\section{Consequences for the estimation of Archimedean and Extreme Value copulas}\label{est}
	
	\subsection{Extreme Value copulas}
	Suppose that $C_A$ is an Extreme Value copula with Pickands function $A$ and suppose that $(X_1,Y_1),(X_2,Y_2),\ldots$ is a random sample from $(X,Y) \sim C_A$. 
	Letting $\hat{A}_n$ denote the CFG estimator according to \cite{Caperaa, InferenceEVs} 
	(for an estimator in the multivariate setting see \cite{GS}) it can be shown that if $A$ is twice continuously differentiable then the corresponding process $\sqrt{n}(\hat{A}_n-A)$ (in the space of $C([0,1],\Vert \cdot \Vert_\infty)$ of all continuous functions on the unit interval) has a weak limit,
	and that, for suitable weight functions, $(\hat{A}_n)_{n \in \mathbb{N}}$ is uniformly, strongly consistent (see \cite[Proposition 4.1]{Caperaa}).	
  Although the estimator $\hat{A}_n$ may fail to be convex in general, following an idea from \cite{InferenceEVs} 
	it can be used to construct a convex estimator $\hat{A}^\ast_n$ given by 
	$$ 
	  \hat{A}^\ast_n
		:= \textrm{greatest convex minorant of } 
		   \max \{ \min\{\hat{A}_n, 1\}, \textrm{id}, 1- \textrm{id} \}
	$$
	where $\textrm{id}$ denotes the identity function on $[0,1]$.
	$\hat{A}^\ast_n$ is a Pickands dependence function (see \cite[Section 3.3]{InferenceEVs}) and the estimator $\hat{A}^\ast_n$ 
	is uniformly, strongly consistent (the latter follows from \cite{Marshall}).
     Hence Theorem \ref{EVcorollary} directly yields weak conditional convergence of the sequence of 
     corresponding Extreme Value copulas $(C_{\hat{A}^\ast_n})_{n \in \mathbb{N}}$ to $C_A$ 
	Moreover, according to Corollary \ref{cor:main2}
	$$
	\lim_{n \rightarrow \infty} \vert \zeta_1(C_{\hat{A}^\ast_n}) - \zeta_1(C_A) \vert = 0
	$$
	holds and the same is true for the dependence measure $r$ studied in \cite{Dette}, i.e., for estimating $\zeta_1(C_A)$ it suffices to have a good estimator of the Pickands dependence function $A$ 
	(we refer to \cite{Griessenberger} for more details concerning the estimation of $\zeta_1(X,Y)$).
	
	\begin{Ex}\label{ex:sim.ev} 
		Consider the Galambos copula $C_A$ with parameter 
		$\theta = 3$, i.e., the Extreme Value copula whose Pickands dependence function is given by 
		$A(x) = 1 - (x^{-3} + (1-x)^{-3})^{-1/3}$. 
		Figure \ref{sample_galambos10000} depicts a sample (and corresponding histograms) 
		for the case $n=10000$. 
		\begin{figure}[!htp]
			\centering
			\includegraphics[width=0.8\textwidth]{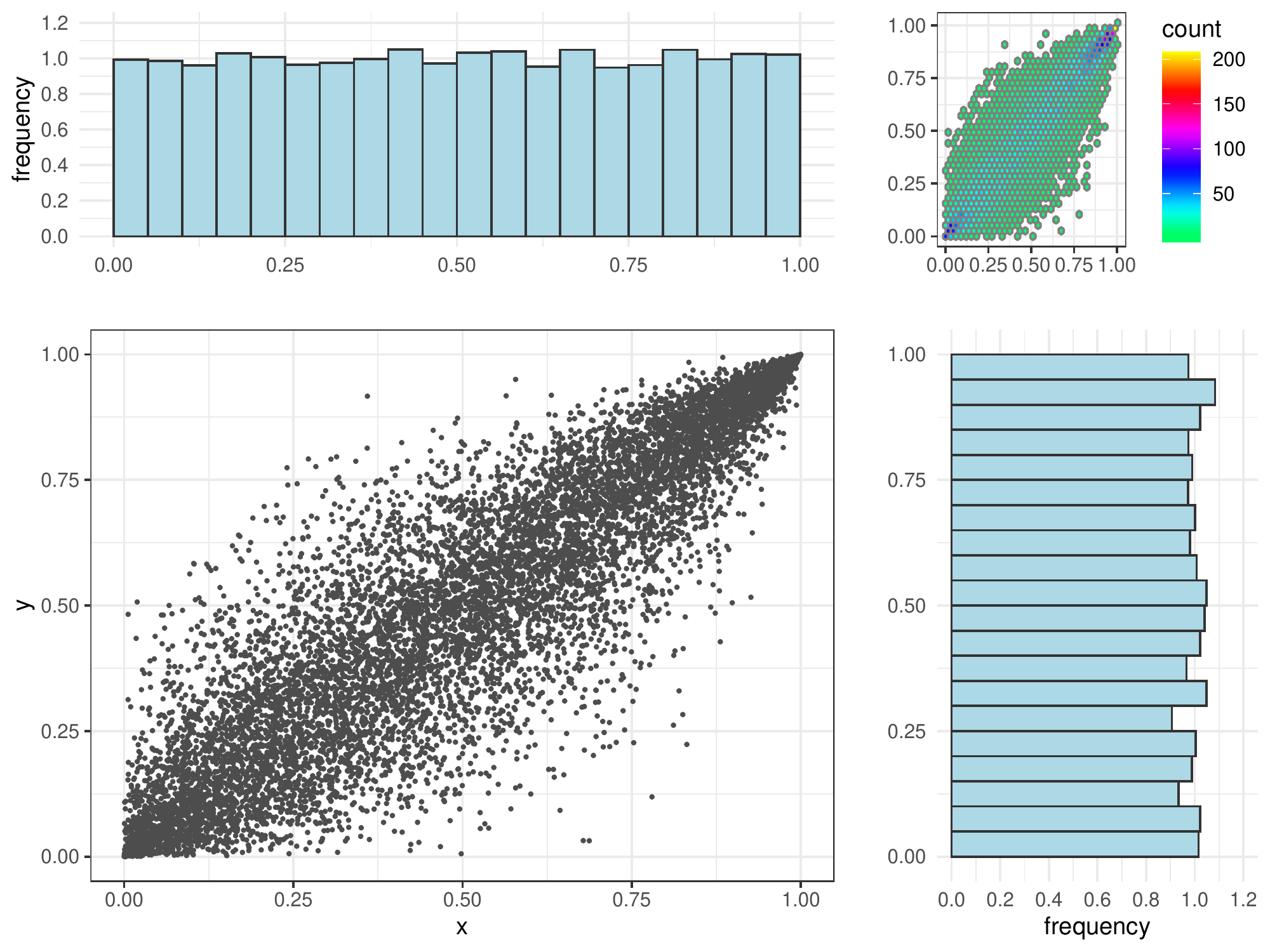}
			\caption{Sample of size $n=10000$ from the Galambos copula with parameter $\theta=3$ as considered in Example \ref{ex:sim.ev} (lower left panel); two-dimensional histogram (upper right panel) and marginal histograms (upper left and lower right panel).}
			\label{sample_galambos10000}
		\end{figure}
		\newline Using the R-package \textquoteleft copula' we calculate the estimator $\hat A_n^\ast$ of $A$ as 
		described above. 
		Figure \ref{estimatedPickands} depicts the obtained generators together with the true Pickands function $A$.
		For the dependence measure $\zeta_1$ (again using the R-package \textquoteleft qad') we obtained the following values: 
		$\zeta_1(C_{\hat A_{500}^\ast}) = 0.7746, \zeta_1(C_{\hat A_{10000}^\ast}) = 0.7594, \zeta_1(C_A) = 0.7513$.
	\end{Ex}
	\begin{figure}[!ht]
			\centering
			\includegraphics[width=0.72\textwidth]{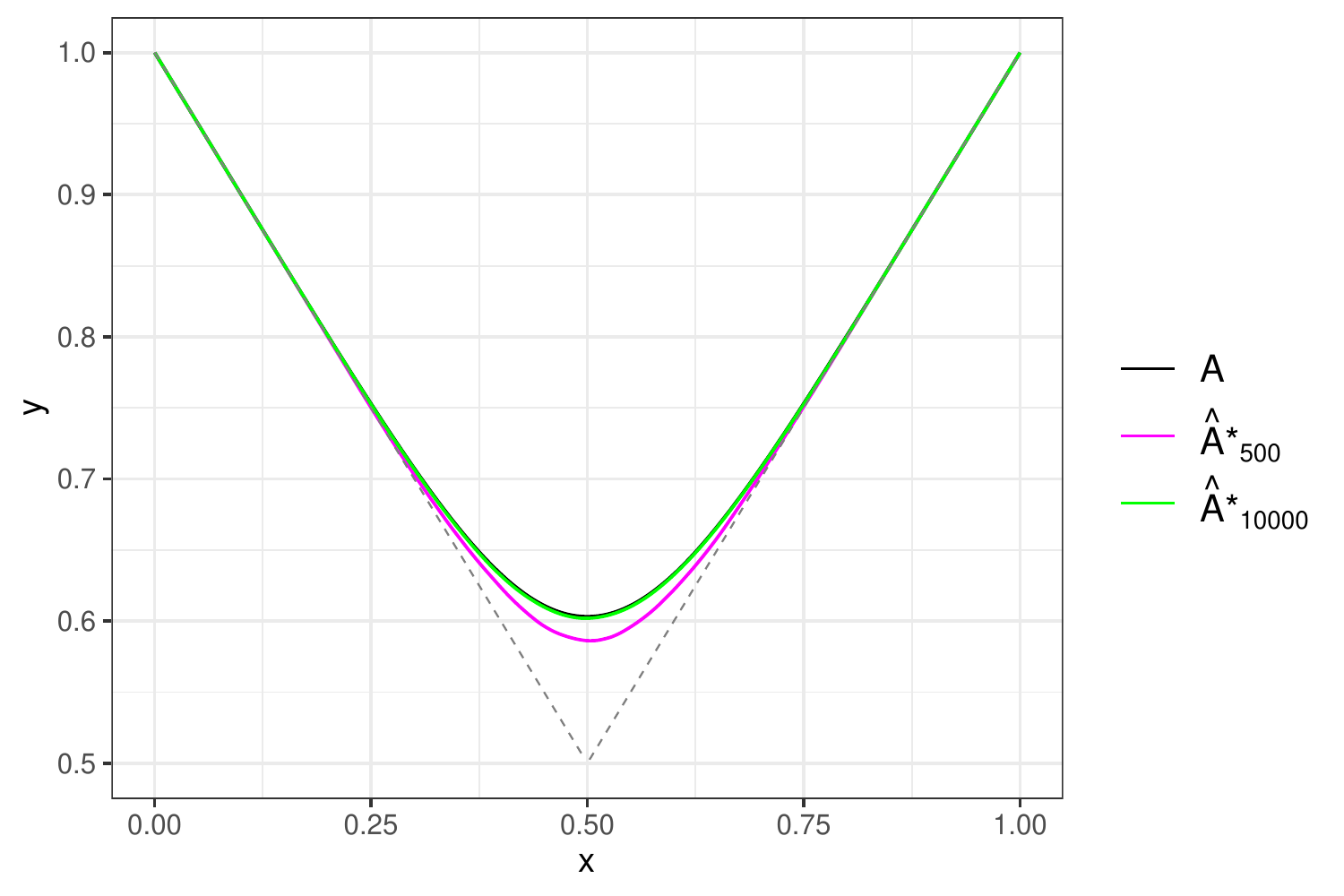}
			\caption{The Pickands dependence function $A$ (black) as well as the estimate $\hat{A}_n^\ast$ for $n=500$ and $n=10000$ as considered in Example \ref{ex:sim.ev}.}
			\label{estimatedPickands}
		\end{figure}
	
	We now focus on the estimation of $\zeta_2 := r$ introduced in \cite{Dette}, denote the estimator of $r$ developed 
	by Chatterjee in \cite{Chatterjee} by $\hat{r}_n$ (for an implementation see the R-package \textquoteleft XICOR' \cite{xicorPackage}) and proceed with a small simulation study comparing the afore-mentioned plugin approach 
	(using the Extreme Value information) 
	with $\hat{r}_n$ (not taking into account the Extreme Value information).	In other words: 
	Given a random sample $(X_1,Y_1)$, $(X_2,Y_2), \ldots$ from $(X,Y)\sim C_A \in \mathcal{C}_{ev}$ we calculate 
	$\hat{r}_n$ and $r(C_{\hat{A}_n^\ast}) = 6\cdot D_2^2(C_{\hat{A}_n^\ast},\Pi)$ for different sample sizes a total of 
	$R=5000$ times. Doing so we consider two cases of $C_A$: the Galambos copula with parameter $\theta = 3$ and
	the Extreme Value copula with piecewise linear Pickands dependence function $A$ given by 
		\begin{align}\label{eq:ex.sim.evc}
		A(x) 
		= \mathbf{1}_{[0,\frac{1}{4}]} (x) \cdot (1-x) + \mathbf{1}_{(\frac{1}{4}, \frac{7}{10}]} \left(-\frac{1}{9} (x-7) \right) + \mathbf{1}_{(\frac{7}{10},1]} (x) \cdot x.
		\end{align}
	\begin{figure}[!ht]
		\centering
		\includegraphics[width=0.95\textwidth]{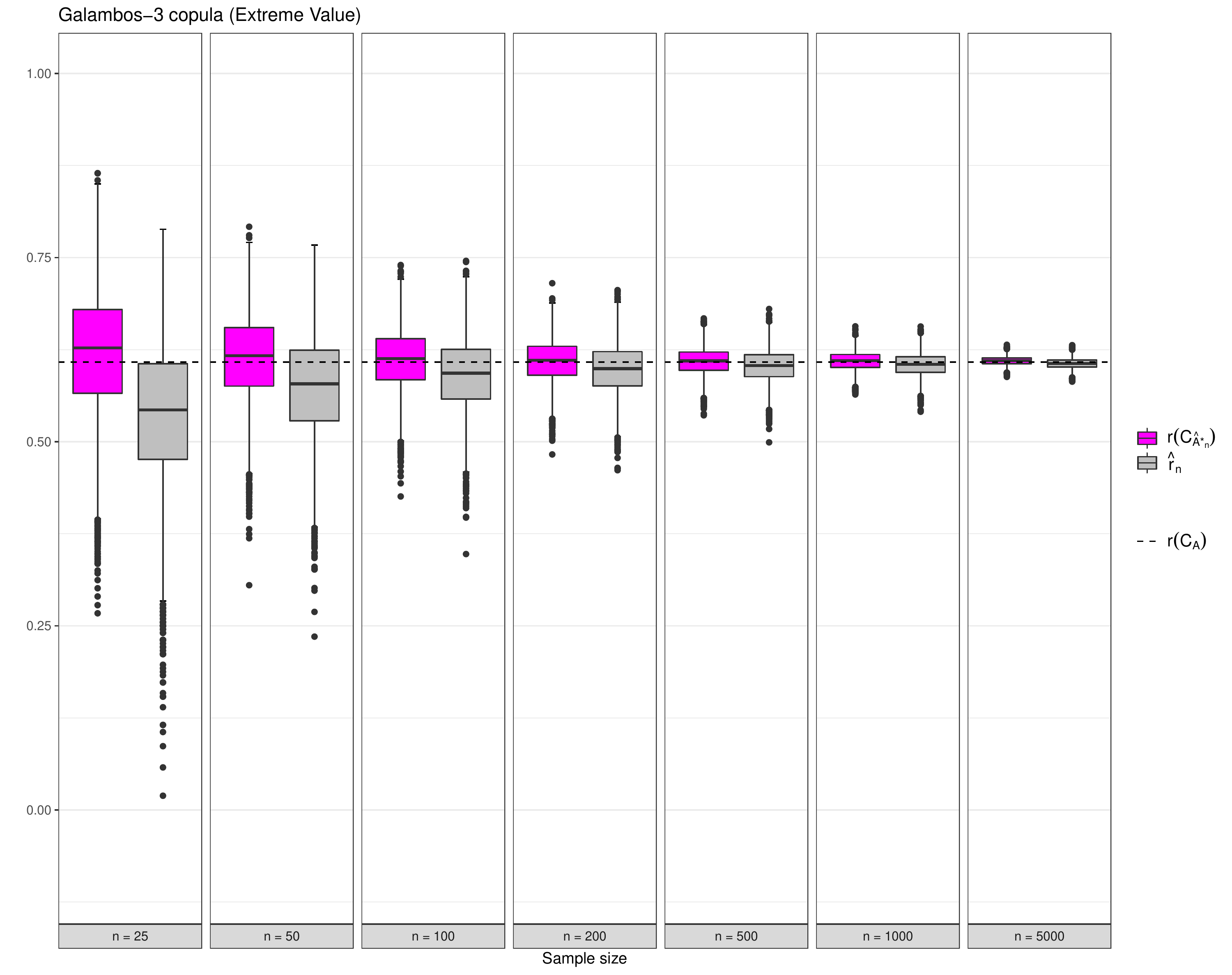}
		\caption{Boxplots of the obtained value of $\hat{r}_n$ and $r(C_{\hat{A}_n^\ast})$ based on samples 
		from the Galambos copula with parameter $\theta = 3$.}
		\label{sim_Gal5}
	\end{figure}				
		\begin{figure}[!ht]
			\centering
			\includegraphics[width=0.95\textwidth]{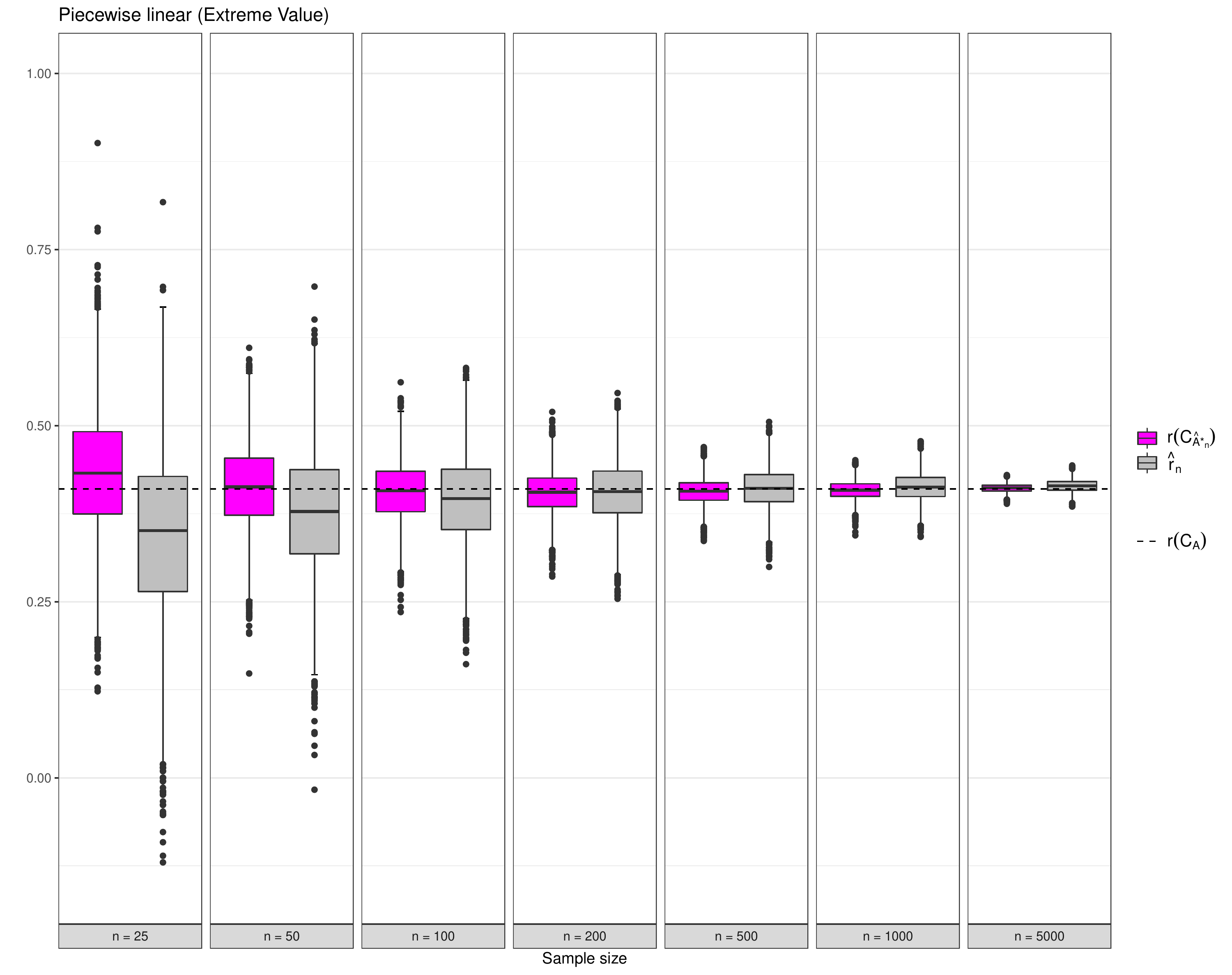}
			\caption{Boxplots of the obtained value of $\hat{r}_n$ and $r(C_{\hat{A}_n^\ast})$ based on
			 samples from the copula with Pickands dependence function $A$ according to eq. (\ref{eq:ex.sim.evc}).}
			\label{sim_EVC_pwl}
		\end{figure}
	
	Not surprisingly, Figure \ref{sim_Gal5} and Figure \ref{sim_EVC_pwl} show that for small to moderate sample sizes 
	the plugin estimator $r(C_{\hat{A}_n^\ast})$ (using the Extreme Value information) yields better results than $\hat{r}_n$, 
	for large sample sizes both estimators perform similarly well.
	
	\subsection{Archimedean copulas}
	Turning to the Archimedean setting suppose now that $C$ is an Archimedean copula with generator $\varphi$ and let $(X_1,Y_1),(X_2,Y_2),\ldots$ be a random 
	sample from $(X,Y) \sim C$. We will let $\hat{F}_n^{\text{Kendall}}$ denote the estimator of the Kendall distribution function $F^{\text{Kendall}}$ of $C$ called $K_{n,2}$ in \cite[Lemma 1]{GNZ}.  
	According to \cite{GNZ} (also see \cite{barbe,GenestInference}), 
	$\hat{F}_n^{\text{Kendall}}$ itself is a Kendall distribution function of an Archimedean copula and	
	under mild regularity conditions the so-called empirical Kendall process 
	$\sqrt{n}(\hat{F}_n^{\text{Kendall}}-F^{\text{Kendall}})$ converges weakly to a centered Gaussian process. 
	If $\hat{F}_n^{\text{Kendall}}$ converges weakly to $F^{\text{Kendall}}$ 
	(to the best of authors' knowledge no sufficient conditions for this property to hold are known in the literature)
	then, according to Theorem \ref{result}, we automatically have weak conditional convergence of the sequence of corresponding Archimedean copulas 
	$(C_{\hat{\varphi}_n})_{n \in \mathbb{N}}$ to $C$, whereby $\hat{\varphi}_n$ denotes the (normalized) generator obtained from 
	$\hat{F}_n^{\text{Kendall}}$. Moreover, according to Corollary \ref{cor:main1}
	$$
	\lim_{n \rightarrow \infty} \vert \zeta_1(C_{\hat{\varphi}_n}) - \zeta_1(C) \vert = 0
	$$
	holds and the same is true for the dependence measure studied in \cite{Dette}, i.e., for estimating $\zeta_1(C)$ it suffices 
	to have a good estimator of the Kendall distribution function $F^{\text{Kendall}}$. 
	
	
	\begin{Ex}\label{ex:estimation}
		We illustrate the afore-mentioned properties with simulations in R and consider the (normalized) generator
		$\varphi(x) = (-\log(\frac{1}{2}))^{-3}\cdot(-\log(x))^3$ of the  Gumbel copula $C_\varphi$ with parameter 
		$\theta = 3$. Figure \ref{sample_gumbel10000} depicts a sample of size
		$n=10000$ from this copula as well as a two-dimensional and the corresponding marginal histograms.
		\begin{figure}[!htp]
			\centering
			\includegraphics[width=0.8\textwidth]{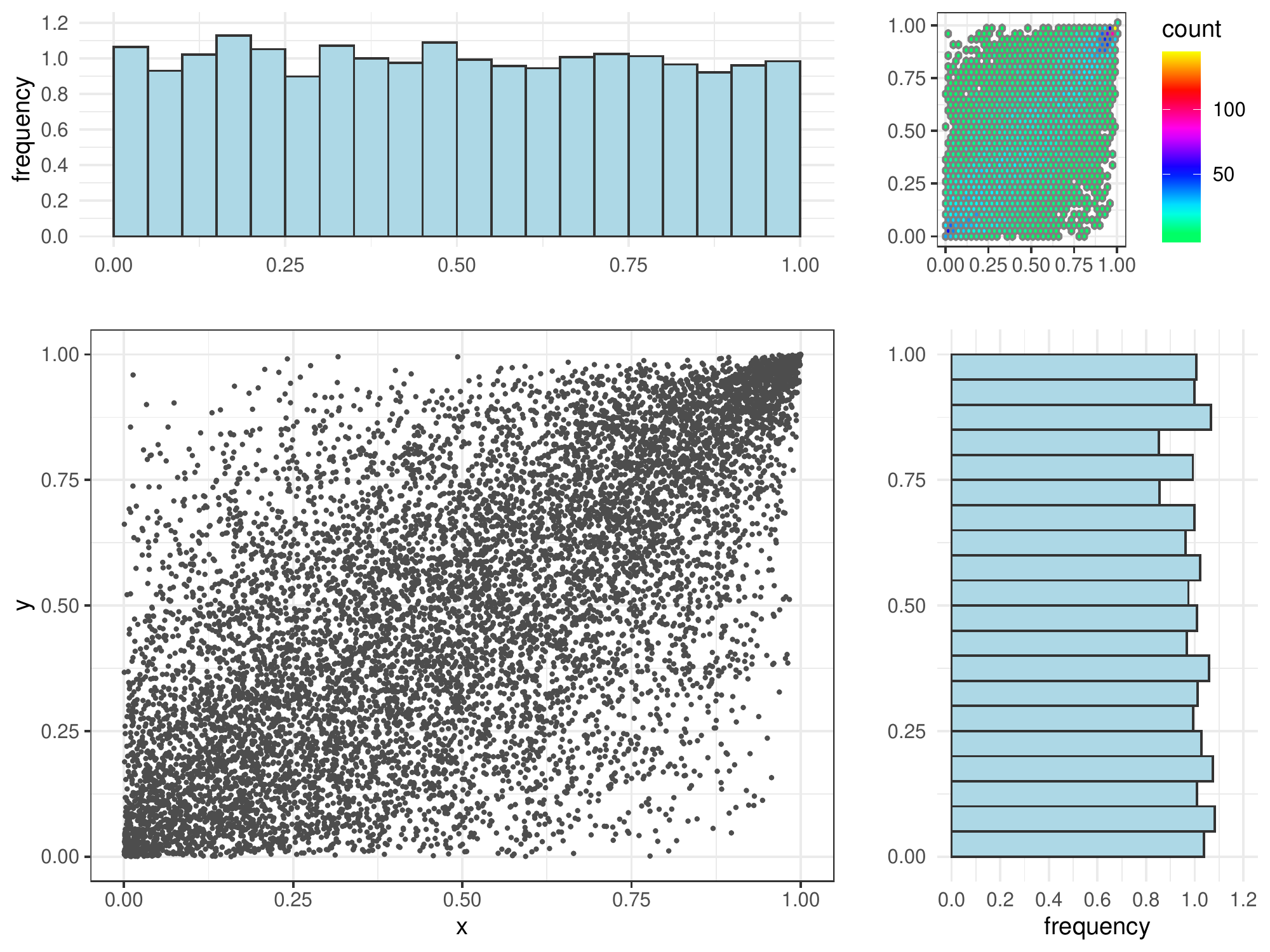}
			\caption{Sample of size $n=10000$ from the Gumbel copula with parameter $\theta=3$ as in Example \ref{ex:estimation} (lower left panel); two-dimensional histogram (upper right panel) and marginal histograms (upper left and lower right panel).}
			\label{sample_gumbel10000}
		\end{figure}
		For both samples we use the R-package \textquoteleft copula' (see \cite{copulaPackage}) to estimate the empi\-rical Kendall distribution function $\hat{F}_n^{\text{Kendall}}$ as described in \cite{GNZ}. Figure \ref{empKendall} (left panel) depicts the real as well as the estimated Kendall distribution function for the sample sizes $n=500$ and $n=10000$, respectively.
		\begin{figure}[!htp]
			\centering
			\includegraphics[width=0.75\textwidth]{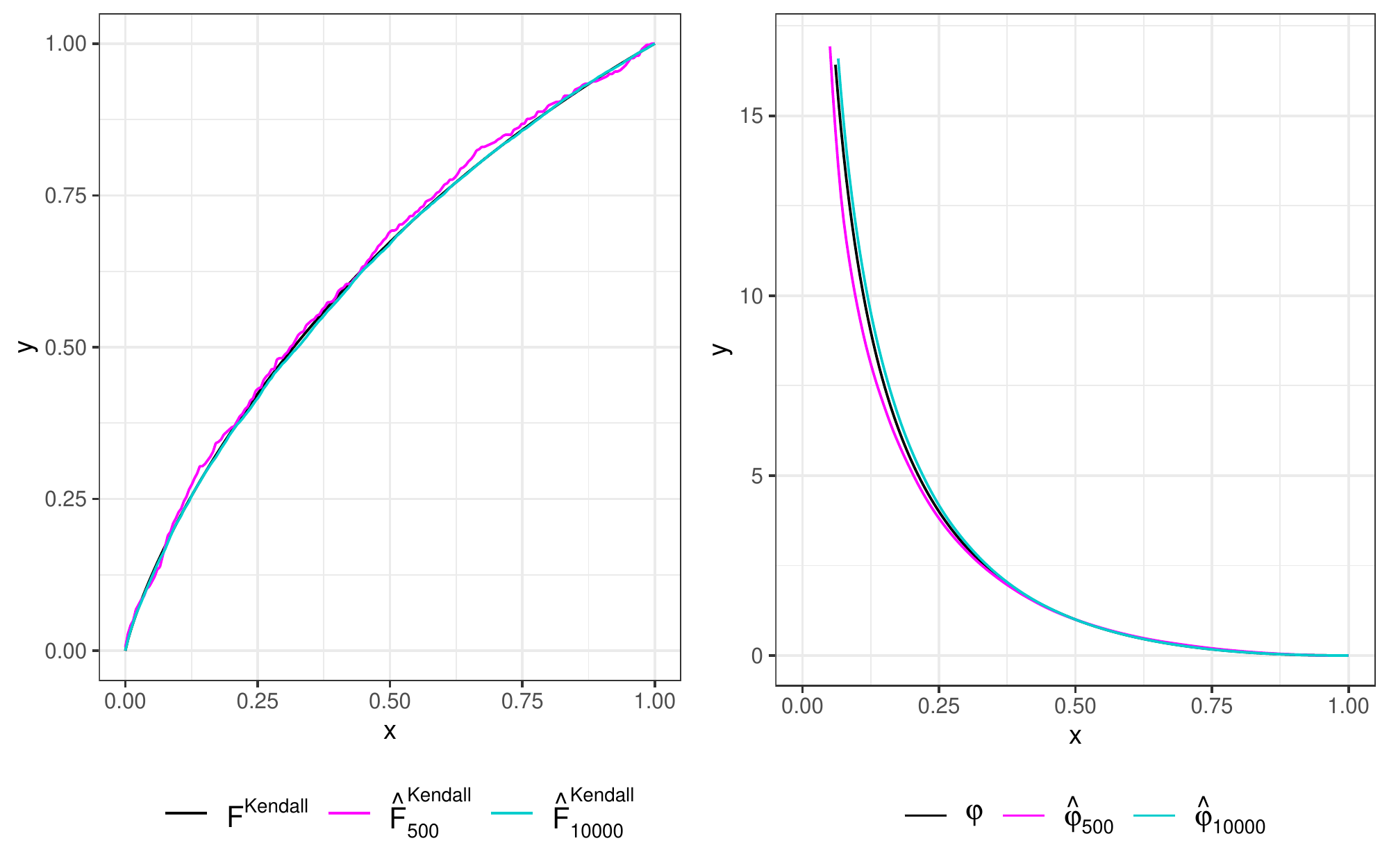}
			\caption{Kendall distribution function of a Gumbel copula with parameter $\theta=3$ (black) 
				and its estimate $\hat{F}_n^{\text{Kendall}}$ for $n=500$ and $n=10000$ as considered in Example \ref{ex:estimation} (left panel); (normalized) generator (black) and its estimates $\hat{\varphi}_n$ 
				for $n=500$ and $n=10000$ as considered in Example \ref{ex:estimation} (right panel).}
			\label{empKendall}
		\end{figure}
		Based on $\hat{F}_n^{\text{Kendall}}$ we derive the estimated (normalized) generator $\hat{\varphi}_n$ by (again see Figure \ref{empKendall})  
		\begin{align*}
		\hat\varphi_n(x) = \exp\bigg(\textrm{sign}\left(x-\frac{1}{2}\right) \int_{\min(\frac{1}{2},x)}^{\max(\frac{1}{2},x)}
		\frac{1}{t-\hat{F}^\text{Kendall}_n(t)} \ \mathrm{d}t\bigg).
		\end{align*}
		Given $\hat{\varphi}_n$ we finally calculate the estimated Archimedean copula $C_{\hat{\varphi}_n}$ and calculate its Markov kernel $K_{C_{\hat{\varphi}_n}}$. 
		For the dependence measure $\zeta_1$ using the R-package \textquoteleft qad' (\cite{qadPackage}) 
		we obtained the following values: $\zeta_1(C_{\hat \varphi_{500}}) = 0.7117, 
		\zeta_1(C_{\hat \varphi_{10000}}) = 0.7041, \zeta_1(C_\varphi) = 0.6910$. 
	\end{Ex}
	
	\begin{figure}[!htp]
		\centering
		\includegraphics[width=0.85\textwidth]{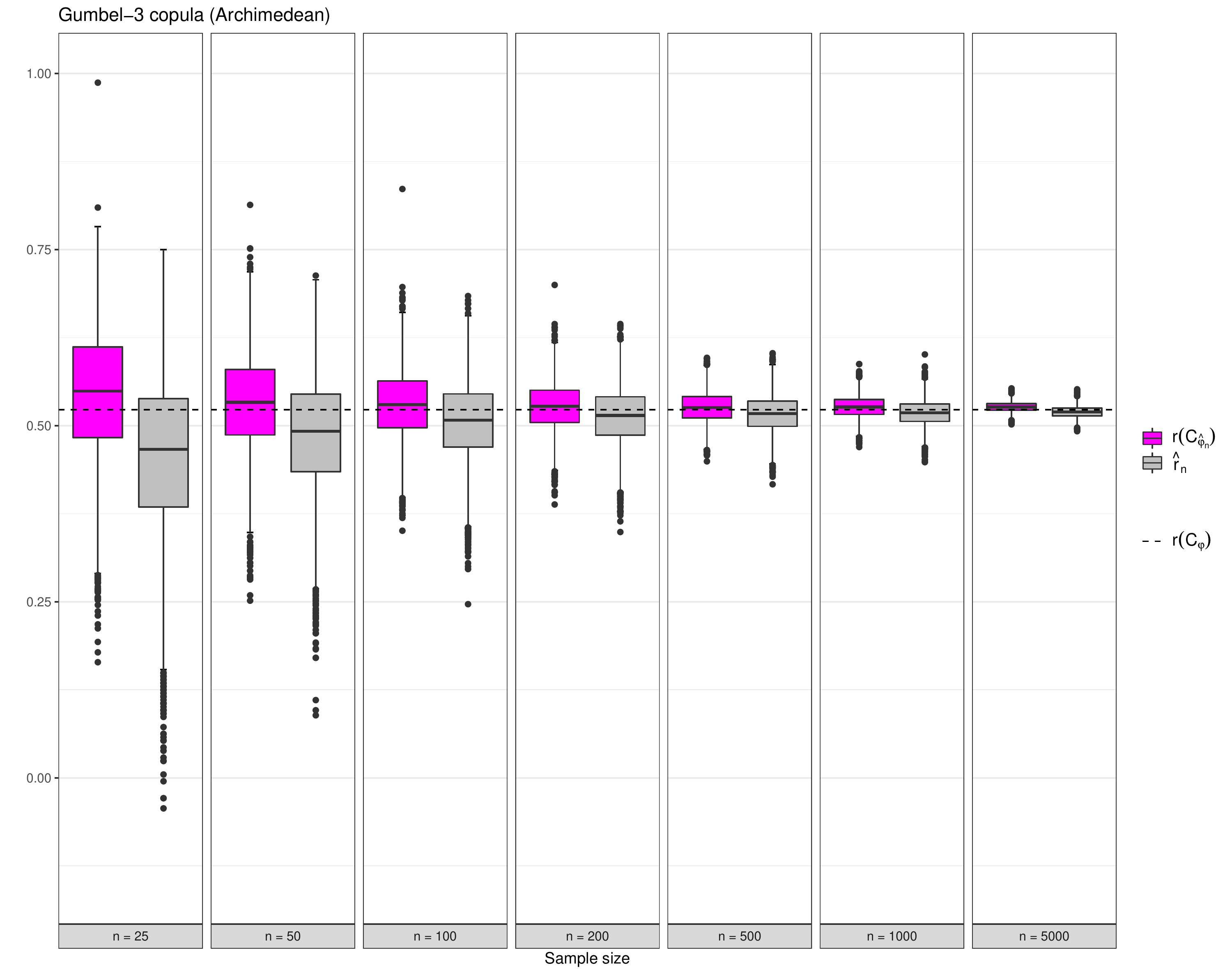}
		\caption{Boxplots of the obtained value of $\hat{r}_n$ and $r(C_{\hat{A}_n^\ast})$ based on samples from 
		the Gumbel copula with parameter $\theta = 3$.}
		\label{sim_Gum3}
	\end{figure}	
	
	As in the Extreme Value setting we perform a small simulation study comparing the performance of the plugin 
	estimator with the estimator $\hat{r}_n$ of the coefficient of correlation $r$ (ignoring the 
	Archimedean information).  
	Generating samples of the Gumbel copula with parameter $\theta = 3$, calculating $\hat{r}_n$ as well as 
	$r(C_{\hat{\varphi}_n})= 6\cdot D_2(C_{\hat{\varphi}_n},\Pi)$ for these samples and repeating 
	this procedure $R=5000$ times yielded the results depicted in Figure \ref{sim_Gum3}. 
	Not surprisingly, for small to moderate sample sizes the plugin estimator $r(C_{\hat{\varphi}_n})$ outperforms 
	the general estimator $\hat{r}_n$, for large sample sizes both estimators perform comparably well.
	
 	\section*{Acknowledgements}
	The first author gratefully acknowledges the financial support from Porsche Holding Austria and 
	Land Salzburg within the WISS 2025 project \textquoteleft KFZ' (P1900123).
	Moreover, the second and the third author gratefully acknowledge the support of the WISS 2025 project \textquoteleft 
	IDA-lab Salzburg' (20204-WISS/225/197-2019 and 0102-F1901166-KZP). \\[-8mm]
	

\begin{thebibliography}{}

\bibitem[\protect\citeauthoryear{Barbe, Genest, Ghoudi, and
  R\'{e}millard}{Barbe et~al.}{1996}]{barbe}
Barbe, P., C.~Genest, K.~Ghoudi, and B.~R\'{e}millard (1996).
\newblock On {K}endall's process.
\newblock {\em J. Multivariate Analysis\/}~{\em 58\/}(2), 197 -- 229.

\bibitem[\protect\citeauthoryear{Billingsley}{Billingsley}{2013}]{billingsley}
Billingsley, P. (2013).
\newblock {\em Convergence of Probability Measures}.
\newblock Wiley Series in Probability and Statistics. Wiley.

\bibitem[\protect\citeauthoryear{Cap\'{e}ra\`{a}, Foug\'{e}res, and
  Genest}{Cap\'{e}ra\`{a} et~al.}{1997}]{Caperaa}
Cap\'{e}ra\`{a}, P., A.-L. Foug\'{e}res, and C.~Genest (1997).
\newblock A nonparametric estimation procedure for bivariate extreme value
  copulas.
\newblock {\em Biometrika\/}~{\em 84}, 567--577.

\bibitem[\protect\citeauthoryear{Charpentier and Segers}{Charpentier and
  Segers}{2008}]{convArch}
Charpentier, A. and J.~Segers (2008).
\newblock Convergence of {A}rchimedean copulas.
\newblock {\em Statist. Probab. Lett.\/}~{\em 78}, 412--419.

\bibitem[\protect\citeauthoryear{Chatterjee}{Chatterjee}{2020a}]{Chatterjee}
Chatterjee, S. (2020a).
\newblock A new coefficient of correlation.
\newblock {\em Journal of the American Statistical Association\/}~{\em 0\/}(0),
  1--21.

\bibitem[\protect\citeauthoryear{Chatterjee}{Chatterjee}{2020b}]{xicorPackage}
Chatterjee, S. (2020b).
\newblock {\em XICOR: Association Measurement Through Cross Rank Increments}.
\newblock {R} package version 0.3.3.

\bibitem[\protect\citeauthoryear{Cherubini, Durante, and Mulinacci}{Cherubini
  et~al.}{2013}]{MarshallOlkinDistributions}
Cherubini, U., F.~Durante, and S.~Mulinacci (2013).
\newblock {\em Marshall-Olkin Distributions - Advances in Theory and
  Applications}.
\newblock Springer International Publishing Switzerland: Springer Proceedings
  in Mathematics \& Statistics.

\bibitem[\protect\citeauthoryear{de~Haan and Resnick}{de~Haan and
  Resnick}{1977}]{deHaan}
de~Haan, L. and S.~Resnick (1977).
\newblock Limit theorem for multivariate sample extremes.
\newblock {\em Z. f{\"u}r Wahrscheinlichkeitstheorie\/}~{\em 40}, 317--337.

\bibitem[\protect\citeauthoryear{Dette, Siburg, and Stoimenov}{Dette
  et~al.}{2013}]{Dette}
Dette, H., K.~Siburg, and P.~Stoimenov (2013).
\newblock A {C}opula-{B}ased {N}on-parametric {M}easure of {R}egression
  {D}ependence.
\newblock {\em Scand. J. Stat.\/}~{\em 40}, 21--41.

\bibitem[\protect\citeauthoryear{Durante and Sempi}{Durante and
  Sempi}{2016}]{Principles}
Durante, F. and C.~Sempi (2016).
\newblock {\em Principles of copula theory}.
\newblock Boca Raton FL: Taylor \& Francis Group LLC.

\bibitem[\protect\citeauthoryear{Fern\'{a}ndez~S\'{a}nchez and
  Trutschnig}{Fern\'{a}ndez~S\'{a}nchez and Trutschnig}{2015a}]{p15}
Fern\'{a}ndez~S\'{a}nchez, J. and W.~Trutschnig (2015a).
\newblock Conditioning based metrics on the space of multivariate copulas,
  their interrelation with uniform and levelwise convergence and {I}terated
  {F}unction {S}ystems.
\newblock {\em J. Theoret. Probab.\/}~{\em 28}, 1311--1336.

\bibitem[\protect\citeauthoryear{Fern\'{a}ndez~S\'{a}nchez and
  Trutschnig}{Fern\'{a}ndez~S\'{a}nchez and Trutschnig}{2015b}]{p21}
Fern\'{a}ndez~S\'{a}nchez, J. and W.~Trutschnig (2015b).
\newblock Singularity aspects of {A}rchimedean copulas.
\newblock {\em J. Math. Anal. Appl.\/}~{\em 432}, 103--113.

\bibitem[\protect\citeauthoryear{Genest, Ne{\v{s}}lehov{\'a}, and
  Ziegel}{Genest et~al.}{2011}]{GNZ}
Genest, C., J.~Ne{\v{s}}lehov{\'a}, and J.~Ziegel (2011).
\newblock {Inference in multivariate Archimedean copula models}.
\newblock {\em Test\/}~{\em 20}, 223.

\bibitem[\protect\citeauthoryear{Genest and Rivest}{Genest and
  Rivest}{1993}]{GenestInference}
Genest, C. and L.~Rivest (1993).
\newblock Statistical inference procedures for bivariate {A}rchimedean copulas.
\newblock {\em J. Amer. Stat. Assoc.\/}~{\em 88}, 1034--1043.

\bibitem[\protect\citeauthoryear{Genest and Segers}{Genest and
  Segers}{2009}]{InferenceEVs}
Genest, C. and J.~Segers (2009).
\newblock Rank-based inference for bivariate extreme-value copulas.
\newblock {\em Annals of Statistics\/}~{\em 37}, 2990--3022.

\bibitem[\protect\citeauthoryear{Ghoudi, Khoudraji, and Rivest}{Ghoudi
  et~al.}{1998}]{Ghoudi}
Ghoudi, K., A.~Khoudraji, and L.-P. Rivest (1998).
\newblock Propri\'{e}t\'{e}s statistiques des copules de valeurs extremes
  bidimensionnelles.
\newblock {\em Canad. J. Statist.\/}~{\em 26}, 187--197.

\bibitem[\protect\citeauthoryear{Griessenberger, Junker, and
  Trutschnig}{Griessenberger et~al.}{2020}]{qadPackage}
Griessenberger, F., R.~R. Junker, and W.~Trutschnig (2020).
\newblock {\em qad: Quantification of Asymmetric Dependence}.
\newblock {R} package version 0.1.2.

\bibitem[\protect\citeauthoryear{Gudendorf and Segers}{Gudendorf and
  Segers}{2010}]{GS2}
Gudendorf, G. and J.~Segers (2010).
\newblock Extreme-value copulas.
\newblock In P.~Jaworski, F.~Durante, K.~H{\"a}rdle, and T.~Rychlik (Eds.), {\em
  Copula Theory and Its Applications}, Chapter~6, pp.\  127--145. Heidelberg:
  Springer, Berlin, Heidelberg.

\bibitem[\protect\citeauthoryear{Gudendorf and Segers}{Gudendorf and
  Segers}{2011}]{GS}
Gudendorf, G. and J.~Segers (2011).
\newblock Nonparametric estimation of an extreme-value copula in arbitrary
  dimensions.
\newblock {\em J. Multivariate Analysis\/}~{\em 102\/}(1), 37 -- 47.

\bibitem[\protect\citeauthoryear{Hofert, Kojadinovic, Maechler, and Yan}{Hofert
  et~al.}{2020}]{copulaPackage}
Hofert, M., I.~Kojadinovic, M.~Maechler, and J.~Yan (2020).
\newblock {\em copula: Multivariate Dependence with Copulas}.
\newblock {R} package version 0.999-20.

\bibitem[\protect\citeauthoryear{Junker, Griessenberger, and Trutschnig}{Junker
  et~al.}{2020}]{Griessenberger}
Junker, R.~R., F.~Griessenberger, and W.~Trutschnig (2020).
\newblock Estimating scale-invariant directed dependence of bivariate
  distributions.
\newblock {\em Computational Statistics \& Data Analysis\/}~{\em 153\/}(0), 0.

\bibitem[\protect\citeauthoryear{Kallenberg}{Kallenberg}{2002}]{Kallenberg}
Kallenberg, O. (2002).
\newblock {\em Foundations of modern probability}.
\newblock New York: Springer-Verlag.

\bibitem[\protect\citeauthoryear{Kannan and Krueger}{Kannan and
  Krueger}{2012}]{kannan}
Kannan, R. and C.~Krueger (2012).
\newblock {\em Advanced Analysis: on the Real Line}.
\newblock Universitext. Springer New York.

\bibitem[\protect\citeauthoryear{Klement, Mesiar, and Pap}{Klement
  et~al.}{2000}]{tnorms}
Klement, E., R.~Mesiar, and E.~Pap (2000).
\newblock {\em Triangular Norms\/} (8 ed.).
\newblock Springer Netherlands.

\bibitem[\protect\citeauthoryear{Klenke}{Klenke}{2008}]{Klenke}
Klenke, A. (2008).
\newblock {\em Wahrscheinlichkeitstheorie}.
\newblock Berlin Heidelberg: Springer Lehrbuch Masterclass Series.

\bibitem[\protect\citeauthoryear{Lancaster}{Lancaster}{1963}]{lanc}
Lancaster, H.~O. (1963, 12).
\newblock Correlation and complete dependence of random variables.
\newblock {\em Ann. Math. Statist.\/}~{\em 34\/}(4), 1315--1321.

\bibitem[\protect\citeauthoryear{Li, Mikusinski, and Taylor}{Li
  et~al.}{1998}]{approximation}
Li, X., P.~Mikusinski, and M.~Taylor (1998).
\newblock Strong approximation of copulas.
\newblock {\em J. Math. Anal. Appl.\/}~{\em 255}, 608--623.

\bibitem[\protect\citeauthoryear{Marshall}{Marshall}{1970}]{Marshall}
Marshall, A.~W. (1970).
\newblock Discussion of barlow and van zwet's papers.
\newblock In M.~L. Puri (Ed.), {\em Nonparametric Techniques in Statistical
  Inference}, pp.\  175--176. London: Cambridge University Press.

\bibitem[\protect\citeauthoryear{Mikusi\'{n}ski and Taylor}{Mikusi\'{n}ski and
  Taylor}{2010}]{approxNCopulas}
Mikusi\'{n}ski, P. and M.~Taylor (2010).
\newblock Some approximations of n-copulas.
\newblock {\em Metrika\/}~(72), 385--414.

\bibitem[\protect\citeauthoryear{Nelsen}{Nelsen}{2006}]{Nelsen}
Nelsen, R. (2006).
\newblock {\em An Introduction to Copulas}.
\newblock Berlin Heidelberg: Springer-Verlag.

\bibitem[\protect\citeauthoryear{Pickands}{Pickands}{1981}]{Pickands}
Pickands, J. (1981).
\newblock Multivariate extreme value distributions.
\newblock {\em Proceedings 43rd Session International Statistical
  Institute\/}~{\em 2}, 859--878.

\bibitem[\protect\citeauthoryear{Pollard}{Pollard}{2001}]{pollard}
Pollard, D. (2001).
\newblock {\em A User's Guide to Measure Theoretic Probability}.
\newblock Cambridge Series in Statistical and Probabilistic Mathematics.
  Cambridge University Press.

\bibitem[\protect\citeauthoryear{Rockafellar}{Rockafellar}{1970}]{rockafellar}
Rockafellar, R. (1970).
\newblock {\em Convex Analysis}.
\newblock Princeton Landmarks in Mathematics and Physics. Princeton University
  Press.

\bibitem[\protect\citeauthoryear{Rudin}{Rudin}{1987}]{rudinrc}
Rudin, W. (1987).
\newblock {\em Real and Complex Analysis, 3rd Ed.}
\newblock USA: McGraw-Hill, Inc.

\bibitem[\protect\citeauthoryear{Sempi}{Sempi}{2004}]{Sempi:Convergence}
Sempi, C. (2004).
\newblock Convergence of copulas: Critical remarks.
\newblock {\em Radovi Matemati\v{c}ki\/}~{\em 12}, 241--249.

\bibitem[\protect\citeauthoryear{Trutschnig}{Trutschnig}{2011}]{p06}
Trutschnig, W. (2011).
\newblock On a strong metric on the space of copulas and its induced dependence
  measure.
\newblock {\em J. Math. Anal. Appl.\/}~{\em 384}, 690--705.

\bibitem[\protect\citeauthoryear{Trutschnig}{Trutschnig}{2012}]{p07}
Trutschnig, W. (2012).
\newblock Some results on the convergence of (quasi-) copulas.
\newblock {\em Fuzzy Sets and Systems\/}~{\em 191}, 113--121.

\bibitem[\protect\citeauthoryear{Trutschnig, Schreyer, and
  Fern\'{a}ndez~S\'{a}nchez}{Trutschnig et~al.}{2016}]{p23}
Trutschnig, W., M.~Schreyer, and J.~Fern\'{a}ndez~S\'{a}nchez (2016).
\newblock Mass distribution of two-dimensional extreme-value copulas and
  related results.
\newblock {\em Extremes\/}~{\em 19}, 405--427.

\end{thebibliography}
	\ifx\undefined\allcaps\def\allcaps#1{#1}\fi

\end{document}